\documentclass[11pt,onecolumn,letterpaper,final]{IEEEtran}
\usepackage{fullpage}

\usepackage[colorlinks=true,citecolor=blue]{hyperref}
\newcommand{\circled}[1]{\small{\raisebox{.6pt}{\textcircled{\raisebox{-.8pt}{#1}}}}}

\usepackage{amssymb}
\usepackage{amsmath,mathrsfs,dsfont}
\usepackage{nicefrac}
\usepackage{algorithm}
\usepackage{algorithmicx}
\usepackage{algpseudocode}

\usepackage{booktabs}

\usepackage{color}
\usepackage{enumitem}

\usepackage{array,cite}
\usepackage{graphicx,tikz}
\usepackage[mathscr]{euscript}
\usepackage{amsthm}
\usepackage{cite}

\usepackage{bm}
\usepackage{bbm}
\usepackage{color}

\usepackage{color}
\usepackage{epstopdf}
\usepackage{subcaption}
\usepackage{cleveref}
\usepackage{thmtools}
\usepackage{thm-restate}

\usepackage{bbding}
\usepackage{mathtools}
\usepackage{slashbox}
\usepackage{multirow}
\allowdisplaybreaks
\usepackage{framed}
\usepackage{xcolor}

\usepackage[margin=1in]{geometry}
\usepackage{setspace}
\singlespace

\title{\Large \bfseries  Sketching for Convex and Nonconvex Regularized Least Squares with Sharp Guarantees }
\author{Yingzhen Yang,\thanks{Yingzhen Yang is with School of Computing and Augmented Intelligence, Arizona State University, Tempe, AZ 85281, USA}
Ping Li
}
\date{}
\graphicspath{{illustrations/}}
\newcounter{optproblem}

\newtheoremstyle{mytheoremstyle} 
    {\topsep}                    
    {\topsep}                    
    {\normalfont}                
    {}                           
    {\bfseries}                   
    {.}                          
    {.5em}                       
    {}  

\theoremstyle{mytheoremstyle}
\newtheorem{theorem}{Theorem}[section]
\newtheorem{remark}[theorem]{Remark}

\newtheorem{corollary}[theorem]{Corollary}

\newtheorem*{theorem*}{Theorem}
\newtheorem*{lemma*}{Lemma}
\newtheorem*{remark*}{Remark}

\newtheorem{lemma}[theorem]{Lemma}

\newtheorem{definition}{Definition}[section]

\newtheorem{assumption}{Assumption}

\DeclareMathAlphabet{\pazocal}{OMS}{zplm}{m}{n}
\DeclareMathAlphabet{\mathpzc}{OMS}{pzc}{m}{it}

\setlist[itemize]{leftmargin=*}

\renewcommand{\hat}{\widehat}

\newcommand{\bfm}[1]{\ensuremath{\mathbf{#1}}}
\newcommand{\bfsym}[1]{\ensuremath{\boldsymbol{#1}}}

   \def\bA{\bfm A}

   \def\bD{\bfm D}  
\def\be{\bfm e}

   \def\bI{\bfm I}

     \def\NN{\mathbb{N}}
     
   \def\bP{\bfm P}  
     
     \def\RR{\mathbb{R}}
\def\bs{\bfm s}   \def\bS{\bfm S}  
\def\bt{\bfm t}     
\def\bu{\bfm u}   \def\bU{\bfm U}  
\def\bv{\bfm v}   \def\bV{\bfm V}  
\def\bw{\bfm w}   \def\bW{\bfm W}  
\def\bx{\bfm x}   \def\bX{\bfm X}  
\def\by{\bfm y}   \def\bY{\bfm Y}  
     
\def\bzero{\bfm 0} 

 \def\cA{{\cal  A}}
 
 \def\cC{{\cal  C}}

 \def\cF{{\cal  F}}

 \def\cI{{\cal  I}}
 \def\cJ{{\cal  J}}
 
 \def\cL{{\cal  L}}
 
 \def\cN{{\cal  N}}
 \def\cO{{\cal  O}}
 \def\cP{{\cal  P}}

 \def\cS{{\cal  S}}
 
 \def\cU{{\cal  U}}
 \def\cV{{\cal  V}}

\def\bbeta{\bfsym \beta}

\def\beps{\bfsym \varepsilon}

\def\bSigma{\bfsym \Sigma}

\def\bOmega {\bfsym {\Omega}}

\def\+#1{\mathcal{#1}}
\def\-#1{\textup{#1}}

\def\set#1{\left\{ #1 \right\}}
\def\pth#1{\left( #1 \right)}

\def\abth#1{\left | #1 \right |}

\def\defeq {\coloneqq}

\newcommand{\La}{\left\langle\kern-0.64ex\left\langle}
\newcommand{\Ra}{\right\rangle\kern-0.64ex\right\rangle}

\def\Norm#1#2{{\left\vert\kern-0.4ex\left\vert\kern-0.4ex\left\vert #1
    \right\vert\kern-0.4ex\right\vert\kern-0.4ex\right\vert}_{#2}}
\def\norm#1#2{{\left\|#1\right\|}_{#2}}

\def\lzeronorm#1{\norm{#1}{0}}
\def\lonenorm#1{\norm{#1}{1}}
\def\ltwonorm#1{\norm{#1}{2}}

\def\supnorm#1{\norm{#1}{\infty}}

\def \supp#1{\textup{supp}\left(#1\right)}

\newcommand{\1}{{\rm 1}\kern-0.25em{\rm I}}
\def\indict#1{{\rm 1}\kern-0.25em{\rm I}_{\set{#1}}}

\DeclarePairedDelimiter\floor{\lfloor}{\rfloor}

\def \eps  {\epsilon}
\def \eps {\varepsilon}

\def \diff {{\rm d}}
\def \iprod#1#2{\left\langle #1, #2 \right\rangle}

\def\set#1{\left\{#1\right\}}

\def\unitsphere#1{\mathbb{S}^{#1}}
\DeclareMathOperator*{\argmin}{arg\,min}

\def \E {\mathbb{E}}
\def\Expect#1#2{\E_{#1}\left[#2\right]}

\def \Pr {\textup{Pr}}
\newcommand{\Prob}[1]{\Pr\left[#1\right]}

\newcommand{\beq}{\begin{equation}}
\newcommand{\eeq}{\end{equation}}
\newcommand{\beqa}{\begin{eqnarray}}
\newcommand{\eeqa}{\end{eqnarray}}
\newcommand{\beqas}{\begin{eqnarray*}}
\newcommand{\eeqas}{\end{eqnarray*}}
\def\bal#1\eal{\begin{align}#1\end{align}}
\def\bals#1\eals{\begin{align*}#1\end{align*}}
\def\bsal#1\esal{\begin{small}\begin{align}#1\end{align}\end{small}}
\def\bsals#1\esals{\begin{small}\begin{align*}#1\end{align*}\end{small}}
\def\bsfal#1\esfal{\begin{small}\begin{flalign}#1\end{flalign}\end{small}}

\def\tbX{{\tilde \bX}}
\begin{document}

\maketitle


\begin{abstract}
Randomized algorithms are important for solving large-scale optimization problems. In this paper, we propose a fast sketching algorithm for least square problems regularized by convex or nonconvex regularization functions, Sketching for Regularized Optimization (SRO). Our SRO algorithm first generates a sketch of the original data matrix, then solves the sketched problem. Different from existing randomized algorithms, our algorithm handles general Frechet subdifferentiable regularization functions in an unified framework. We present general theoretical result for the approximation error between the optimization results of the original problem and the sketched problem for regularized least square problems which can be convex or nonconvex. For arbitrary convex regularizer, relative-error bound is proved for the approximation error. Importantly, minimax rates for sparse signal estimation by solving the sketched sparse convex or nonconvex learning problems are also obtained using our general theoretical result under mild conditions. To the best of our knowledge, our results are among the first to demonstrate minimax rates for convex or nonconvex sparse learning problem by sketching under a unified theoretical framework. We further propose an iterative sketching algorithm which reduces the approximation error exponentially by iteratively invoking the sketching algorithm. Experimental results demonstrate the effectiveness of the proposed SRO and Iterative SRO algorithms.
\end{abstract}


\section{Introduction}

Efficient optimization by randomized algorithms is an important topic in machine learning and optimization, and it has broad applications in numerical linear algebra, data analysis and scientific computing. Randomized algorithms based on matrix sketching or random projection have received a lot of attention \cite{Vempala2004-random-projection,Boutsidis2009-random-projection-nonnegative-LSE,Drineas2011-fast-least-square-approximation,Mahoney-random-matrix-data,Kane2014-sparser_JL-transforms}, which solve sketched problems of much smaller scale. Sketching algorithms has been used to approximately solve various large-scale problems including least square regression, robust regression, low-rank approximation, singular value decomposition and matrix factorization \cite{Halko11-find-structure-randomness,LuDFU13-faster-ridge-regression,Alaoui15-fast-kernel-ridge-regression,RaskuttiM16-sketching-OLS,Yang0JZ15a-dual-sparse-randomized-reduction,
Drineas16-RandNLA,Oymak18-isometric-sketching,Oymak17-randomized-dimension-reduction,Tropp2017-low-rank-sketching}. On the other hand, regularized problems with convex or nonconvex regularization, such as the well-known $\ell^{1}$ or $\ell^{2}$-norm regularized least square estimation, also known as Lasso or ridge regression, play essential roles in machine learning and statistics. While most existing research works demonstrate the potential of random projection and sketching on problems with common convex regularization \cite{ZhangWZ2016-sparse-linear-random-projection} or convex constraints \cite{Pilanci2016-IHS}, few efforts are made in the analysis of regularized problems with general convex or nonconvex regularization.

In this paper, we study efficient sketching algorithm for a general class of optimization problems with convex or nonconvex regularization, which is presented as follows:
\bal\label{eq:optimization-general}
&\min_{\bbeta \in \RR^d} f(\bbeta) = \frac{1}{2} \ltwonorm{\by - \bX \bbeta}^2+ h_{\lambda}(\bbeta).
\eal%
$\bX \in \RR^{n \times d}$ is the data matrix or design matrix for regression problems, $h_{\lambda} \colon \RR^d \to \RR$ is a regularizer function and $\lambda$ is a positive regularization weight. When $h_{\lambda}(\cdot) = \lambda \ltwonorm{\cdot}^2$ or $h_{\lambda}(\cdot) = \lambda \lonenorm{\cdot}$, (\ref{eq:optimization-general}) is the optimization problem for ridge regression or $\ell^1$ regularized least square estimation (Lasso).

We study the regime that $n \gg r = {\rm rank}(\bX)$ where $r$ is the rank of $\bX$ in most results of this paper, and it is a popular setting for large-scale problems such as fast least square estimation by sketching \cite{Drineas2011-fast-least-square-approximation}. For example, matrix $\bX$ in sparse linear regression is usually low-rank or approximately low-rank in practice. However, our results for sparse nonconvex learning in Subection~\ref{sec:sparse-nonconvex-learning} hold for $\bX$ not necessarily low-rank.

Optimization for (\ref{eq:optimization-general}) is time consuming when $n$ is large, and such large-scale regularized optimization problems are important due to increasing interest in massive data. To this end, we propose Sketching for Regularized Optimization (SRO) in this paper as an efficient randomized algorithm for problem (\ref{eq:optimization-general}). With $\tilde n < n$ where $\tilde n$ is the target row number of a sketch of the data matrix $\bX$, SRO first generates a sketched version of $\bX$ by $\tbX =  \bP \bX$, then solves the following sketched problem,
\bal\label{eq:optimization-general-rp}
&\min_{\bbeta \in \RR^d} \tilde f(\bbeta) = \frac{1}{2} \bbeta^{\top} {\tbX}^{\top} \tbX \bbeta - \iprod{\by}{\bX \bbeta} + h_{\lambda}(\bbeta).
\eal%

One hopes that the optimization result of the sketched problem (\ref{eq:optimization-general-rp}), denoted by $\tilde \bbeta^*$, is a good approximation to that of the original problem (\ref{eq:optimization-general}), denoted by $\bbeta^*$. The optimization and theoretical computer science literature are particulary interested in the solution approximation measure defined as the semi-norm induced by the data matrix $\bX$, i.e. $\norm{\tilde \bbeta^* - \bbeta^*}{\bX} = \ltwonorm{\bX (\tilde \bbeta^* - \bbeta^*) }$ where $\norm{\bu}{\bX} \defeq \ltwonorm{\bX \bu}$ for any vector $\bu$. Existing research, such as Iterative Hessian Sketch (IHS) \cite{Pilanci2016-IHS}, prefers relative-error approximation to the solution of the original problem in the following form:
\bal\label{eq:relative-error-approximation}
&\norm{\tilde \bbeta^* - \bbeta^*}{\bX} \le \rho \norm{\bbeta^*}{\bX},
\eal%
where $0 < \rho < 1$ is a positive constant. With the relative-error approximation (\ref{eq:relative-error-approximation}), IHS proposes an interesting iterative sketching method to reduce the approximation error $\norm{\tilde \bbeta^* - \bbeta^*}{\bX}$ geometrically in the iteration number.

\subsection{Contributions and Main Results}

First,  we prove that SRO for a general class of optimization problems in the form of (\ref{eq:optimization-general}) enjoys an universal approximation error bound for arbitrary regularization function $h$ which is Frechet subdifferentiable, that is,
\bal\tag{Theorem~\ref{theorem::error-bound-general}}
&(1-\eps) \norm{\tilde \bbeta^* - \bbeta^*}{\bX}^2 - \eps \norm{ \tilde \bbeta^* - \bbeta^*}{\bX} \norm{\bbeta^*}{\bX} \le Q_{h_{\lambda}},
\eal%
where $\eps \in (0,1)$ is a small positive number which appears in the approximation error bound to be presented. $Q_h$ is a quantity depending on the degree of nonconvexity of $h$. In particular, $Q_h = 0$ for arbitrary convex regularization function $h$, leading to a preferred relative-error approximation essential to Iterative SRO to be explained soon, that is,
\bal\label{eq:error-bound-convex-intro}
&\norm{\tilde \bbeta^* - \bbeta^*}{\bX}  \le \frac{\eps}{1-\eps} \norm{\bbeta^*}{\bX}.
\eal%
With a small positive $\eps$, the relative-error approximation (\ref{eq:relative-error-approximation}) is achieved. Furthermore, if $h$ is strongly convex, we prove in Theorem~\ref{theorem::error-bound-strongly-convex} that the relative-error approximation also applies to $\ltwonorm{\tilde \bbeta^* - \bbeta^*}^2$, the gap between the actual optimization results of the original problem and the sketched problem. If $h$ is nonconvex but not ``very'' nonconvex with limited $Q_h$, we prove that sketching still admits a form of relative-error approximation bound in Corollary~\ref{corollary::error-bound-nonconvex}. Our theoretical results convey a clear message that relative-error approximation ``favors'' convex regularization, and show how convexity and nonconvexity of regularization affect the accuracy of sketching by the quantity $Q_{h_{\lambda}}$ .

Based on SRO, we present an iterative sketching algorithm termed Iterative SRO which provably reduces the approximation error $\norm{ \tilde \bbeta^* - \bbeta^*}{\bX}$ geometrically in an iterative manner. Albeit being bounded, the approximation error of one-time SRO is still not small enough for various applications. To this end, Iterative SRO iteratively calls SRO to approximate the residual of the last iteration so as to further reduce the approximation error, given the relative-error approximation (\ref{eq:relative-error-approximation}).
More details are introduced in Section~\ref{sec::iterative-SRO} and Section~\ref{sec::experiments}.

Second, we study the sparse signal estimation problem by sketching in Section~\ref{sec:sketching-sparse-signal-recovery}. Using the general result in Theorem~\ref{theorem::error-bound-general}, for sparse convex or nonconvex learning problems where $h_{\lambda}$ is convex or nonconvex,
we obtain minimax rates of the order $\cO\pth{\sqrt{ \bar s\log d/n}}$ for sparse signal estimation by solving the sketched problem where $\bar s$ is the support size of the unknown sparse parameter vector to be estimated. Our proofs are based on the general result in Theorem~\ref{theorem::error-bound-general}. To the best of our knowledge, our analysis provides the first unified theoretical result for sharp error rates for sparse signal estimation using sketching based optimization method.

There are two key differences between Iterative SRO and IHS \cite{Pilanci2016-IHS}. First, using the subspace embedding as the projection matrix $\bP$, Iterative SRO does not need to sample $\bP$ and compute the sketched matrix $\tbX = \bP \bX$ at each iteration, in contrast with IHS where a separate $\bP$ is sampled and $\tbX$ is computed at each iteration. This advantage saves considerable computation and storage for large-scale problems. Second, while IHS is restricted to constrained least-square problems with convex constraints, SRO and Iterative SRO are capable of handling all convex regularization and certain nonconvex regularization in a unified framework. For example, we show that Generalized Lasso~\cite{tibshirani2011-generalized-lasso} can be efficiently and effectively solved by Iterative SRO in Section~\ref{sec::experiments}.

\subsection{Notations}

Throughout this paper, we use bold letters for matrices and vectors, regular lower letters for scalars. The bold letter with subscript indicates the corresponding element of a matrix or vector, and the bold letter with superscript indicates the corresponding column of a matrix, i.e. $\bX^i$ indicates the $i$-th column of matrix $\bX$. $\norm{\cdot}{p}$ denotes the $\ell^{p}$-norm of a vector, or the $p$-norm of a matrix. $\sigma_{t}(\cdot)$ is the $t$-th largest singular value of a matrix, and $\sigma_{\min}(\cdot)$ and $\sigma_{\max}(\cdot)$ indicate the smallest and largest singular value of a matrix respectively. ${\rm tr}(\cdot)$ is the trace of a matrix.  $f_1(n) = \Theta (f_2(n))$ if there exist constants $k_1,k_2>0$ and $n_0$ such that $k_1 f_(n) \le f_1(n) \le k_2 f_2(n)$. We use $\bX  \succcurlyeq \bY$ to indicate that $\bX - \bY$ is a positive semi-definite matrix, and $\bI_d$ indicates the $d \times d$ identity matrix. ${\rm rank}(\bX)$ means the rank of a matrix $\bX$. $\NN$ denotes the set of all the natural numbers, and we use $[m \ldots n]$ to indicate numbers between $m$ and $n$ inclusively, and $[n]$ denotes the natural numbers between $1$ and $n$ inclusively. ${\textsf {nnz}}(\bX)$ indicates the number of nonzero elements of a matrix $\bX$. $\abth{\cdot}$ denotes the cardinality of a set, and $\supp{\cdot}$ denotes the set of indices of nonzero elements for a vector.

\section{The SRO Algorithm}\label{sec::algorithm}

In order to improve the efficiency of optimization for (\ref{eq:optimization-general}), we propose Reguarlzied Optimization by Sketching (SRO) in this section. The key idea is to sketch matrix $\bX$ in the quadratic term of (\ref{eq:optimization-general}) by random projection. It is comprised of two steps:
\begin{itemize}[leftmargin=*,label={}]
\setlength{\leftskip}{-5pt}
\item Step 1. Project the matrix $\bX$ onto a lower dimensional space by a linear transformation $\bP \in \RR^{{\tilde n} \times n}$ with $\tilde n < n$, i.e. $\tbX = \bP \bX$. $\tilde n$ is named the sketch size.

\item Step 2. Solve the sketched problem (\ref{eq:optimization-general-rp}).

\end{itemize}
Sketching the matrix $\bX$ only in the quadratic term $\ltwonorm{\bX \bbeta}^2$ is proposed in \cite{Pilanci2016-IHS} for constrained least square problems with convex constraints. SRO adopts this idea for regularized~least square problems admitting a broad range of regularizers.

The linear transformation $\bP$ is required to be a subspace embedding \cite{Woodruff2014-skeching-numerical-algebra} defined in Definition~\ref{def:subspace-embedding}. The literature \cite{Frankl:1987-JL-Lemma,Indyk1998-ANN,Zhang2016-sparse-random-convex-concave} extensively studies such random transformation which is also closely related to the proof of the Johnson-Lindenstrauss lemma \cite{Dasgupta2003-JL-proof}.
\begin{definition}\label{def:subspace-embedding}
Suppose $\cP$ is a distribution over $\tilde n \times n$ matrices, where $\tilde n$ is
a function of $n$, $d$, $\eps$, and $\delta$. Suppose that with probability at least $1-\delta$, for
any fixed $n \times d$ matrix $\bX$, a matrix $\bP$ drawn from distribution $\cP$\ has the
property that $\bP$ is a $(1 \pm \eps)$ $\ell^2$-subspace embedding for $\bX$, that is,
\bal\label{eq:subspace-embedding}
&(1-\eps)\ltwonorm{\bX\bbeta}^2 \le \ltwonorm{\bP\bX\bbeta}^2 \le  (1+\eps)\ltwonorm{\bX\bbeta}^2
\eal%
holds for all $\bbeta \in \RR^d$. Then we call $\cP$ an $(\eps,\delta)$ oblivious $\ell^2$-subspace embedding.
\end{definition}

\begin{definition}
\label{def:gaussian-subspace-embedding}
(Gaussian Subspace Embedding, \cite[Theorem 2.3]{Woodruff2014-skeching-numerical-algebra}) Let $0 < \eps, \delta < 1$, $\bP = \frac{\bP'}{\sqrt{\tilde n}}$ where $\bP^{'} \in \RR^{\tilde n \times n}$ is a matrix whose elements are i.i.d. samples from the standard Gaussian distribution $\cN(0,1)$. Then if $\tilde n = \cO((r+\log{\frac{1}{\delta}})\eps^{-2})$, for any matrix $\bX \in \RR^{n \times d}$ with $r = {\rm rank}(\bX)$, with probability $1-\delta$, $\bP = \frac{\bP'}{\sqrt{\tilde n}}$ is a $(1 \pm \eps)$ $\ell^2$-subspace embedding for $\bX$. $\bP$ is named a Gaussian subspace embedding.
\end{definition}
\begin{definition}\label{def:sparse-subspace-embedding}
\label{def:sparse-subspace-embedding}
{\rm (Sparse Subspace Embedding)} Let $\bP \in \RR^{\tilde n \times n}$. For each $i \in [n]$, $h(i) \in [\tilde n]$ is uniformly chosen from $[\tilde n]$, and $\sigma(i)$ is a uniformly random element of $\{1,-1\}$. We then set $\bP_{h(i)i}=\sigma(i)$ and set $\bP_{ji} = 0$ for all $j \neq i$. As a result, $\bP$ has only a single nonzero element per column, and it is called a sparse subspace embedding.
\end{definition}
Lemma~\ref{lemma::sparse-subspace-embedding} below, also presented in \cite{Clarkson2013-sparse-subspace-embedding}, shows that the sparse subspace embedding defined above is indeed a subspace embedding with a high probability.
\begin{lemma}
[\hspace{-.1pt}{\cite[Theorem 2.1]{Clarkson2013-sparse-subspace-embedding}}]
\label{lemma::sparse-subspace-embedding}
Let $\bP \in \RR^{\tilde n \times n}$ be a sparse embedding matrix with $\tilde n = \cO(r^2/(\delta\eps^2))$ rows. Then for any fixed $n \times d$ matrix $\bX$ with $r = {\rm rank}(\bX)$, with probability $1-\delta$, $\bP$ is a $(1 \pm \eps)$ $\ell^2$-subspace embedding for $\bX$. Furthermore, $\bP \bX$ can be computed in $\cO({\rm nnz}(\bX))$ time, where ${\rm nnz}(\bX)$ is the number of nonzero elements of $\bX$.
\end{lemma}

\subsection{Error Bounds}\label{sec::summary-error-bounds}

The solutions to the original problem (\ref{eq:optimization-general}) and the sketched problem (\ref{eq:optimization-general-rp}) by typical iterative optimization algorithms, such as gradient descent for smooth $h$ or proximal gradient method for non-smooth $h$, are always critical points of the corresponding objective functions under mild conditions \cite{BoltePAL2014}. Therefore, the analysis in the gap between $\tilde \bbeta^*$ and $\bbeta^*$ amounts to the analysis in the distance between critical points of the objective functions of (\ref{eq:optimization-general-rp}) and that of (\ref{eq:optimization-general}), which is presented in Section~\ref{sec::error-bounds}. In the sequel, $\tilde \bbeta^*$ is a critical point of the objective function (\ref{eq:optimization-general-rp}) and $\bbeta^*$ is a critical point of the objective function (\ref{eq:optimization-general}), if no confusion arises. More details about optimization algorithms are deferred to supplementary.

\section{Approximation Error Bounds}\label{sec::error-bounds}

We present an universal approximation error bound for SRO on the general problem (\ref{eq:optimization-general}) in Section~\ref{sec::general-bound}. We then apply this universal result to problems with convex, strongly convex and nonconvex regularization and derive the corresponding relative-error approximation bounds. Before stating our results, the definition of Frechet subdifferential, critical point, strong convexity and degree of nonconvexity are introduced below, which are essential to our analysis.
\begin{definition}
\label{def:subdifferential-critical-points}
{\rm (Subdifferential and critical points)}
\normalfont
Given a nonconvex function $f \colon \RR^d \to \RR \cup \{+\infty\}$ which is a proper and lower semi-continuous function,
\begin{itemize}[leftmargin=*]
\item for a given $\bx \in {\rm dom}f$, its Frechet subdifferential of $f$ at $\bx$, denoted by $\tilde \partial f(\bx)$, is
the set of all vectors $\bu \in \RR^d$ which satisfy
\begin{small}\begin{align*}
&\liminf\limits_{\by \neq \bx,\by \to \bx} \frac{f(\by)-f(\bx)-\langle \bu, \by-\bx \rangle}{\ltwonorm{\by-\bx}} \ge 0.
\end{align*}\end{small}%
\item The limiting-subdifferential of $f$ at $\bx \in \RR^d$, denoted by $\partial f(\bx)$, is defined by
\begin{small}\begin{align*}
\partial f(\bx) = \{\bu \in \RR^d &\colon \exists \bx^k \to \bx, f(\bx^k) \to f(\bx), \tilde \bu^k \in {\tilde \partial f}(\bx^k) \to \bu\}.
\end{align*}\end{small}%
\end{itemize}
The point $\bx$ is a critical point of $f$ if $\bzero \in \partial f(\bx)$.
\end{definition}

Note that Frechet subdifferential generalizes the notions of Frechet derivative and subdifferential of convex functions. If $f$ is a convex function, then $\tilde \partial f(\bx)$ is also the subdifferential of $f$ at $\bx$. In addition, if $f$ is a real-valued differentiable function with gradient $\nabla f(\bx)$, then $\tilde \partial f(\bx) = \{\nabla f(\bx)\}$.

\begin{definition}\label{def:strong-semi-convex}
\normalfont
{\rm (Strongly convex function)}
A differential function $h \colon \RR^d \to \RR$ is $\sigma$-strongly convex for $\sigma > 0$ if for any $\by, \bx \in \RR^d$,
\bal\label{eq:strongly-convex}
&h_{\lambda}(\by) \ge h_{\lambda}(\bx) + \langle \nabla h_{\lambda}(\bx), \by-\bx \rangle + \frac{\sigma}{2} \ltwonorm{\by - \bx}^2.
\eal%
\end{definition}
In order to analyze the relative-error approximation bound, we need the following definition of degree of nonconvexity in terms of the Frechet subdifferential in Definition~\ref{def:subdifferential-critical-points}. It is an extension of the univariate degree of nonconvexity presented in \cite{zhang2012} used to analyze the consistency of nonconvex sparse estimation models with concave regularization.

\begin{definition}\label{def:degree-nonconvexity}
\normalfont
The degree of nonconvexity of a function $h \colon \RR^d \to \RR$ at a point $\bt \in \RR^d$ is defined as
\bal\label{eq:degree-nonconvexity}
\theta_{h}(\bt,\kappa) \defeq \sup_{\bs \in \RR^d, \bs \neq \bt, \bu \in {\tilde \partial h}(\bs), \bv \in {\tilde \partial h}(\bt)}  \frac{-(\bs-\bt)^{\top} (\bu - \bv) - \kappa \ltwonorm{\bs-\bt}^2}{\ltwonorm{\bs - \bt}},
\eal%
where $\kappa \in \RR$. We abbreviate (\ref{eq:degree-nonconvexity}) as $\theta_{h}(\bt,\kappa) \triangleq \sup_{\bs \in \RR^d, \bs \neq \bt} \{ -\frac{1}{\ltwonorm{\bs - \bt}}(\bs-\bt)^{\top} ({\tilde \partial h}(\bs) - {\tilde \partial h}(\bt)) - \kappa \ltwonorm{\bs-\bt}\}$ in the following text.
\end{definition}
\begin{remark}
\normalfont
It can be verified that the degree of nonconvexity of any convex function $h$ is zero with $\kappa = 0$, that is, $\theta_{h}(\bt,0) \le 0$ when $h$ is convex. If $h$ is $\sigma$-strongly convex and $h$ is twice continuously differentiable, then $\nabla^2 h(\bx) \succcurlyeq \sigma \bI_d$ for any $\bx \in \RR^d$, and $\theta_{h}(\bt,-\sigma) \le 0$ with $\kappa$ set to $-\sigma$ in (\ref{eq:degree-nonconvexity}). 
\end{remark}

\subsection{General Bound}\label{sec::general-bound}

\begin{theorem}\label{theorem::error-bound-general}
\normalfont
Suppose $\tilde \bbeta^*$ is any critical point of the objective function in (\ref{eq:optimization-general-rp}), and $\bbeta^*$ is any critical point of the objective function in (\ref{eq:optimization-general}). Suppose $0 < \eps < \eps_0 < 1$ where $\eps_0$ is a small positive constant, $0 < \delta < 1$, $\bP$ is drawn from an $(\eps,\delta)$ oblivious $\ell^2$-subspace embedding over $\tilde n \times n$ matrices. Then with probability $1 - \delta$,
\bal\label{eq:error-bound-general}
(1-\eps) \norm{ \tilde \bbeta^* - \bbeta^*}{\bX}^2 -  \eps \norm{ \tilde \bbeta^* - \bbeta^*}{\bX} \norm{\bbeta^*}{\bX} \le \theta_{h_{\lambda}}(\bbeta^*,\kappa) \ltwonorm{\tilde \bbeta^* - \bbeta^*} + \kappa \ltwonorm{\tilde \bbeta^* - \bbeta^*}^2.
\eal%
In particular, if $\bP$ is a Gaussian subspace embedding, then $\tilde n = \cO((r+\log{\frac{1}{\delta}})\eps^{-2})$. If $\bP$ is a sparse subspace embedding, then $\tilde n = \cO(r^2/(\delta\eps^2))$.
\end{theorem}

$\kappa$ in the definition of degree of nonconvexity and the general approximation error bound (\ref{eq:error-bound-general}) reflects how ``nonconvexity'' of $h$ affects the accuracy of sketching. The following corollary shows that if the Frechet subdifferential of $h$ is Lipschitz continuous with limited Lipschitz constant, or $\kappa$ equivalently, the degree of nonconvexity can be set to $0$, rendering a desirable relative-error approximation bound.

\begin{corollary}\label{corollary::error-bound-nonconvex}
\normalfont
Under the conditions of Theorem~\ref{theorem::error-bound-general}, suppose $\bX$ has full column rank and $h_{\lambda}$ is nonconvex. If the Frechet subdifferential of $ h_{\lambda}$ is $L_h$-smooth, i.e.
$\sup_{\bu \in \tilde \partial h_{\lambda}(\bx), \bv \in \tilde \partial h_{\lambda}(\by)}\ltwonorm{\bu - \bv} \le L_h \ltwonorm{\bx-\by}$. If
$\sigma_{\min}^2(\bX) > L_h$ and $0 < \eps <  1- \frac{L_h}{\sigma_{\min}^2(\bX)}$, then \bal\label{eq:error-bound-noncovnex}
& \norm{ \tilde \bbeta^* - \bbeta^*}{\bX} \le  \frac{\eps }{(1-\eps) - \frac{L_h}{\sigma_{\min}^2(\bX)} } \cdot \norm{\bbeta^*}{\bX}.
\eal%
\end{corollary}
\begin{remark}
To the best of our knowledge, (\ref{eq:error-bound-noncovnex}) is among the very few results in theoretical guarantee of sketching for nonconvex regularization.
If $h$ is twice continuously differentiable, the condition about $L_h$-smoothness is reduced to $\nabla^2 h_{\lambda} (\bbeta) \succcurlyeq - L_h \bI_d$ for all $\bbeta \in \RR^d$. The corollary above states that if $h$ is nonconvex but not ``very'' nonconvex with a limited $L_h$ or the regularization weight $\lambda$ is small enough, i.e. $ L_h < {\sigma_{\min}^2 (\bX)} (1-\eps)$, then we have the relative-error approximation bound (\ref{eq:error-bound-noncovnex}).
\end{remark}

\subsection{Relative-Error Approximation Bound with Convex Regularization}

It can be verified that the degree of nonconvexity vanishes with $\kappa = 0$ when $h$ is convex. As a result, we have relative-error approximation bound for $\norm{\tilde \bbeta^* - \bbeta^*}{\bX}$ shown in Theorem~\ref{theorem::error-bound-convex} below.

\begin{theorem}\label{theorem::error-bound-convex}
\normalfont
If $h_{\lambda}$ is convex, then under the conditions of Theorem~\ref{theorem::error-bound-general}, with probability $1 - \delta$,
\bal\label{eq:error-bound-convex}
&\norm{\tilde \bbeta^* - \bbeta^*}{\bX}  \le \frac{ \eps}{1-\eps} \norm{\bbeta^*}{\bX}.
\eal%
\end{theorem}
When $h_{\lambda}$ is strongly convex, more fine-grained approximation error bound in terms of the actual gap between $\tilde \bbeta^*$ and $\bbeta^*$ is presented in the following theorem.
\begin{theorem}\label{theorem::error-bound-strongly-convex}
If $h$ is $\sigma$-strongly convex, then under the conditions of Theorem~\ref{theorem::error-bound-general}, with probability $1 - \delta$,
\bal\label{eq:error-bound-strongly-convex}
&\ltwonorm{\tilde \bbeta^* - \bbeta^*}^2 \le \frac{\eps^2}{4 \sigma (1-\eps)} \cdot  \norm{\bbeta^*}{\bX}^2,
\eal%
and when $\norm{\tilde \bbeta^* - \bbeta^*}{\bX} \neq 0$,
\bal\label{eq:error-bound-strongly-convex-stronger-appro}
\norm{\tilde \bbeta^* - \bbeta^*}{\bX}  \le \frac{ \eps}{1-\eps} \norm{\bbeta^*}{\bX} - \frac{\sigma \ltwonorm{\tilde \bbeta^* - \bbeta^*}^2}{(1-\eps) \norm{\tilde \bbeta^* - \bbeta^*}{\bX}}.
\eal%
\end{theorem}
\begin{remark}
Comparing (\ref{eq:error-bound-strongly-convex-stronger-appro}) with (\ref{eq:error-bound-convex}), we observe that sketching for strongly-convex regularization renders tighter approximation error than that for the general convex case when $\norm{\tilde \bbeta^* - \bbeta^*}{\bX} \neq 0$. This indicates the benefit of using strongly-convex regularizers for sketching.
\end{remark}
\section{Iterative SRO}\label{sec::iterative-SRO}
\begin{algorithm}[!htb]
\caption{Iterative SRO}
\label{alg:iterative-sketch-SRO}
\begin{algorithmic}

\State Input:
Initialize $\bbeta^{(0)} = \bzero$, iteration number $N > 0$, $t=0$.
\State \textbf{for } $t \gets 1$ to $N$
\State Set
\begin{align}
\bbeta^{(t)}  = \argmin_{\bbeta \in \RR^d} \frac{1}{2} \ltwonorm{\tbX \pth{\bbeta - \bbeta^{(t-1)}} }^2 - \iprod{\by - \bX \bbeta^{(t-1)}}{\bX\bbeta}  + h_{\lambda}(\bbeta )
\end{align}%
\State \textbf{end for}
\State Return $\bbeta^{(N)}$
\end{algorithmic}
\end{algorithm}
Inspired by Iterative Hessian Sketch \cite{Pilanci2016-IHS}, we introduce an iterative sketching method for SRO so that the gap between solutions to the original problem and the sketched problem can be further reduced. The key idea is to iteratively apply SRO to generate a sequence $\{\bbeta^{(t)}\}_{t=1}^N$ such that $\bbeta^{(t)}$ is a more accuracy approximation to $\bbeta^*$, the solution to the original problem (\ref{eq:optimization-general}), than $\bbeta^{(t-1)}$. Consider the optimization problem
\bal\label{eq:optimization-general-translated}
&\min_{\bbeta \in \RR^d} \frac{1}{2} \ltwonorm{ \bX (\bbeta + \bbeta^{(t-1)})}^2 - \by^{\top} \bX \bbeta + h_{\lambda}(\bbeta + \bbeta^{(t-1)}),
\eal%
then $\bbeta^* - \bbeta^{(t-1)}$ is an optimal solution to (\ref{eq:optimization-general-translated}). We apply SRO to problem (\ref{eq:optimization-general-translated}) and suppose $\hat \bbeta$ is an solution to the sketched problem, i.e.
\bal\label{eq:optimization-general-translated-rp}
\hat \bbeta = \argmin_{\bbeta \in \RR^d} \frac{1}{2} \ltwonorm{ \tbX \bbeta}^2 - \iprod{\by - \bX \bbeta^{(t-1)}}{\bX\bbeta} + h_{\lambda}(\bbeta + \bbeta^{(t-1)}).
\eal%
$\hat \bbeta$ is supposed to be an approximation to $\bbeta^* - \bbeta^{(t-1)}$. If $\hat \bbeta$ admits the relative-error approximation bound (\ref{eq:relative-error-approximation}), then $\bbeta^{(t)} = \hat \bbeta + \bbeta^{(t-1)}$ becomes a more accurate approximation to $\bbeta^*$ than $\bbeta^{(t-1)}$ by a factor of $\rho$. This can be verified by noting that $\norm{\bbeta^{(t)} - \bbeta^*}{\bX} = \norm{\hat \bbeta - (\bbeta^*-\bbeta^{(t-1)})}{\bX} \le \rho \norm{\bbeta^*-\bbeta^{(t-1)}}{\bX}$. By mathematical induction, we have Theorem~\ref{theorem::iterative-sketch} below showing that the approximation error of Iterative SRO, which is formally described by Algorithm~\ref{alg:iterative-sketch-SRO}, drops geometrically in the iteration number. It should be emphasized that Theorem~\ref{theorem::iterative-sketch} also handles certain nonconvex regularization.
\begin{theorem}\label{theorem::iterative-sketch}
\normalfont
Under the conditions of Theorem~\ref{theorem::error-bound-general}, with probability at least $1-\delta$ with $\delta \in (0,1)$, the output of Iterative SRO described by Algorithm~\ref{alg:iterative-sketch-SRO} satisfies
\bal\label{eq:iterative-sketch}
&\norm{\bbeta^{(N)} - \bbeta^*}{\bX}  \le \rho^N \norm{\bbeta^*}{\bX}
\eal%
for a constant $0< \rho < 1$ if $h$ is convex, or the Frechet subdifferential of $h$ is $L_h$-smooth and $\bX$ has full column rank with $\frac{L_h}{\sigma_{\min}^2(\bX)} < (1-\eps)$. In particular, if $\bP$ is a Gaussian subspace embedding, then $\tilde n = \cO\pth{\pth{r+\log{\frac{1}{\delta}}}\cdot(\rho+1)^2/\rho^2}$. If $\bP$ is a sparse subspace embedding, then $\tilde n = \cO\pth{r^2/\delta \cdot(\rho+1)^2/\rho^2}$. Here $r = {\rm rank} (\bX)$.
\end{theorem}

\section{Sketching for Sparse Signal Estimation}
\label{sec:sketching-sparse-signal-recovery}
We study sparse signal estimation by sketching in this section.
We consider the linear model widely used in the sparse signal estimation literature, $\bar \by = \bar \bX \bar \bbeta + \beps$ where $\beps$ is a noise vector of i.i.d. sub-gaussian elements with variance proxy $\sigma^2$, and $\bar \bbeta$ is the sparse parameter vector of interest. Following the standard analysis for parameter estimation in the literature such as \cite{YangWLEZ16-sparse-nonlinear-regression,Zhang10-sparse-multistage-convex}, we assume $\max_{i \in [d]} \ltwonorm{\bar \bX^i} \le \sqrt{n}$, and it follows that $\max_{i \in [d]} \ltwonorm{\bX^i} \le 1$. The statistical learning literature has extensively studied the approximation to $\bar \bbeta$ by the M-estimator obtained as a globally or locally optimal solution to problem (\ref{eq:optimization-general}) with $\bX = \bar \bX/\sqrt{n}$, $\by = \bar \by/\sqrt{n}$. That is, one hopes to approximate $\bar \bbeta$ by the globally or locally optimal solution to problem (\ref{eq:optimization-general}) with a suitable sparsity-inducing regularizer $h_{\lambda}$. In Subsection~\ref{sec:sparse-convex-learning}, we show that the Iterative SRO described in Algorithm~\ref{alg:iterative-sketch-SRO} achieves the minimax parameter estimation error of the order $\sqrt{{\bar s} \log d/n}$ where ${\bar s} = \lzeronorm{\bar \bbeta}$. In Subsection~\ref{sec:sparse-nonconvex-learning}, we prove that SRO achieves the minimax parameter estimation error of the order $\sqrt{{\bar s} \log d/n}$ for sparse nonconvex learning, where the nonconvex regularizer $h_{\lambda}$ is the sum of a concave penalty function $q_{\lambda}$ and $\lambda \lonenorm{\cdot}$.

\subsection{Sketching for Sparse Convex Learning}
\label{sec:sparse-convex-learning}
  We define $\cL(\bbeta) \defeq 1/2\cdot \bbeta^{\top}  \bX^{\top} \bX \bbeta - \by^{\top} \bX \bbeta$ and $\tilde \cL(\bbeta) \defeq {1/2\cdot\tbX}^{\top} \tbX \bbeta - \by^{\top} \bX \bbeta$. We introduce the following definition of sparse eigenvalues widely used in sparse signal estimation literature.
\begin{definition}\label{def:sparse-eigenvalue}
(Sparse Eigenvalues)
Let $s$ be a positive integer. The largest and smallest $s$-sparse eigenvalues of the Hessian matrix $\nabla^2 \cL(\bbeta) = \bX^{\top} \bX$ is
\bal
\rho_{\cL,+}(s) &\defeq \sup\set{\bv^{\top} \bX^{\top} \bX \bv \colon
\lzeronorm{\bv} \le s, \ltwonorm{\bv}=1, \bv \in \RR^d}, \\
\rho_{\cL,-}(s) &\defeq \inf\set{\bv^{\top} \bX^{\top} \bX \bv \colon
\lzeronorm{\bv} \le s, \ltwonorm{\bv}=1, \bv \in \RR^d}.
\eal
$\rho_{\tilde \cL,+}(\cdot)$ and $\rho_{\tilde \cL,-}(\cdot)$ are defined in a similar manner with $\bX$ replaced by $\tilde \bX$.
\end{definition}

The following assumption is frequently used in the sparse signal estimation literature with the convex sparsity-inducing penalty, $\lambda \lonenorm{\cdot}$.
\begin{assumption}
(Assumption in \cite{YangWLEZ16-sparse-nonlinear-regression,Zhang10-sparse-multistage-convex} for sparse signal estimation)
\label{assumption:convex-sparse-recovery}
$\rho_{\cL,+}(s) < \infty, \rho_{\cL,-}(s) > 0$ are positive constants. Moreover,
for ${\bar s} = \lzeronorm{\bar \bbeta}$, there exists a $k^* \in \NN$ such that
$k^* \ge 2 {\bar s}$ and
\bal\label{eq:convex-sparse-recovery-sparse-eigenvalue-cond}
\rho_{\cL,+}(k^*) / \rho_{\cL,-}(2k^*+{\bar s}) \le 1+0.5 k^*/{\bar s}.
\eal
\end{assumption}
We study the sparse signal estimation problem by solving the sketched Lasso problem with $h_{\lambda}(\bbeta) = \lambda \lonenorm{\bbeta}$ in
the original problem (\ref{eq:optimization-general}) and the sketched problem (\ref{eq:optimization-general-rp}) using our Iterative SRO algorithm. We have the following sharp bound for the parameter estimation error.
\begin{theorem}\label{theorem:sketching-convex-minimax}
Suppose Assumption~\ref{assumption:convex-sparse-recovery} holds. Let $\lambda = c \sigma \sqrt{\log d/n}$ where $c$ is a positive constant.
Suppose Algorithm~\ref{alg:iterative-sketch-SRO} returns $\tilde \bbeta^* = \bbeta^{(N)}$ with $\rho \in (0,1)$ in (\ref{eq:iterative-sketch}), and the iteration number $N$ is chosen as $N  = 1 + \log\pth{\ltwonorm{\bX} \norm{\bbeta^*}{\bX}/(\lambda \mu)}/\log(1/\rho)$. Then with probability at least $1-\delta-2/d$ with $\delta \in (0,1)$,
\bal\label{eq:sketching-convex-minimax}
\ltwonorm{\tilde \bbeta^* - \bar \bbeta} \le \frac{\pth{c+ \mu c+ 2}
  }{\rho_{\cL,-}({\bar s}+k^*) \cdot \pth{1- \gamma \sqrt{0.5}}  } \sqrt{\frac{\bar s\log d}{n}},
\eal
where $\mu$ is a positive constant, $\gamma =(1+ \mu + 2/c)/(1-\mu - 2/c)$, and $\mu$ and $c$ are chosen such that $\gamma \sqrt{0.5} < 1$.
In particular, if $\bP$ is a Gaussian subspace embedding, then $\tilde n = \cO\pth{\pth{r+\log{\frac{1}{\delta}}}\cdot(\rho+1)^2/\rho^2}$. If $\bP$ is a sparse subspace embedding, then $\tilde n = \cO\pth{r^2/\delta \cdot(\rho+1)^2/\rho^2}$. Here $r = {\rm rank} (\bX)$.
\end{theorem}

Theorem~\ref{theorem:sketching-convex-minimax} shows that our Iterative ROS described in Algorithm~\ref{alg:iterative-sketch-SRO} applied on the sketched problem (\ref{eq:optimization-general-rp}) achieves the parameter estimation error, which is $\ltwonorm{\tilde \bbeta^* - \bar \bbeta}$, of the order $\sqrt{{\bar s} \log d/n}$. Such estimation error rate is not improvable and it is the minimax error rate for standard Lasso. With the rank $r \ll n$, Iterative ROS obtains the solution $\tilde \bbeta^* = \bbeta^{(N)}$ efficiently with a small sketch size $\tilde n$ and the iteration number $N$ is only of a logarithmic order.

\subsection{Sketching for Sparse Nonconvex Learning}
\label{sec:sparse-nonconvex-learning}
We now study sparse signal estimation by sparse nonconvex learning where the regularizer $h_{\lambda}$ is nonconvex in
the original problem (\ref{eq:optimization-general}) and the sketched problem (\ref{eq:optimization-general-rp}), that is, $h_{\lambda}(\bbeta) = \lambda \lonenorm{\bbeta}+Q_{\lambda}(\bbeta)$ where $Q_{\lambda}(\bbeta)
\defeq \sum\limits_{j=1}^d q_{\lambda}(\beta_j)$, $q_{\lambda}$ is a concave function and $\beta_j$ is the $j$-th element of $\bbeta$. We have $h_{\lambda}(\bbeta) = \sum\limits_{j=1}^d \pth{\lambda \abth{\beta_j} + q_{\lambda}(\beta_j)}$. Following the analysis of sparse parameter vector recovery in \cite{Wang2014-sparse-nonconvex},
$\lambda \abth{\cdot} + q_{\lambda}(\cdot)$ is a nonconvex function which can be either
smoothly clipped absolute deviation (SCAD) \cite{Fan2001-SCAD} or minimax
concave penalty (MCP) \cite{Zhang2010-MCP}. More details about the nonconvex regularizer $h_{\lambda}$ are deferred to Section~\ref{sec:nonconvex-penalty}
of the supplementary.
The following regularity conditions on the concave function $q_{\lambda}$ are used in \cite{Wang2014-sparse-nonconvex}.
\begin{assumption}
(Regularity Conditions on Nonconvex Penalty in \cite{Wang2014-sparse-nonconvex} for sparse signal
recovery)
\label{assumption:q-regularity-cond}
\begin{itemize}[leftmargin=18pt]
\setlength{\itemsep}{1pt}
\setlength\itemsep{0.1em}
\item[(a)] $q'_{\lambda}(\beta_j)$ is monotone and Lipschitz continuous. For $\beta'_j > \beta_j$, there exist two constants $\zeta_{-} \ge 0,\zeta_{+} \ge 0$ such that
$-\zeta_{-} \le \frac{q'_{\lambda}(\beta'_j) - q'_{\lambda}(\beta_j)}
{\beta'_j-\beta_j} \le -\zeta_+$.
\item[(b)] $q_{\lambda}(-\beta_j) = q_{\lambda}(\beta_j)$ for all $\beta_j \in \RR$. Also, $q_{\lambda}(0) = q'_{\lambda}(0) = 0$.
\item[(c)]  $q'_{\lambda}(\beta_j) \le \lambda$ for all $\beta_j \in \RR$, and $\abth{q'_{\lambda_1}(\beta_j) - q'_{\lambda_2}(\beta_j)} \le \abth{\lambda_1
-\lambda_2}$ for all $\lambda_1 > 0, \lambda_2 > 0$.
\end{itemize}
\end{assumption}
The following assumption is the standard assumption in \cite{Wang2014-sparse-nonconvex} for sparse signal estimation with the minimax error rate, that is, $\ltwonorm{\bbeta^* - \bar \bbeta} \le \cO\pth{\sqrt{\bar s \log d/{n}}}$.
\begin{assumption}
(Assumption in \cite{Wang2014-sparse-nonconvex} for sparse signal estimation)
\label{assumption:nonconvex-sparse-recovery}
Let ${\bar s} = \lzeronorm{\bar \bbeta}$. There exist an integer
$\tilde s$ such that $\tilde s > C {\bar s}$ such that
$\rho_{\cL,+}({\bar s}+2\tilde s) < \infty, \rho_{\cL,-}({\bar s}+2\tilde s) > 0$ are two absolute constants. The concavity parameter $\zeta_{-}$ satisfies
$\zeta_{-} \le C' \rho_{\cL,-}({\bar s}+2\tilde s)$ with constant $C' \in (0,1)$.
Here $C = 144\kappa^2 + 250\kappa $ with
$\kappa = (\rho_{\cL,+}({\bar s}+2\tilde s) - \zeta_+)/(\rho_{\cL,-}({\bar s}+2\tilde s) - \zeta_{-})$.
\end{assumption}

The following corollary shows that when the support size of $\tilde \bbeta^* - \bbeta^*$ is bounded by $s_0$, then the nonconvex sparse learning problem (\ref{eq:optimization-general}) with $h_{\lambda}$ being the nonconvex regularizer specified in this subsection still enjoys relative-error approximation. It is worth noting that this corollary  follows from our general
result in Theorem~\ref{theorem::error-bound-general}, and it is employed to prove our main result of minimax estimation error rate by sketching in Theorem~\ref{theorem:sketching-nonconvex-minimax}.

\begin{corollary}\label{corollary::error-bound-noncovnex-sparse-eigenvalue}
Under the conditions of Theorem~\ref{theorem::error-bound-general} and
Assumption~\ref{assumption:nonconvex-sparse-recovery}, let $\bP \in \RR^{\tilde n \times n}$ be a Gaussian subspace embedding and suppose that
$\abth{\supp{\tilde \bbeta^* - \bbeta^*} \bigcup
\supp{\bbeta^*} \bigcup \supp{\tilde \bbeta^*}} \le s_0$ for some integer $s_0 \in [d]$. Let $\tilde n \ge c_0 \eps^{-2}\pth{\log(2/\delta) + s_0 \log d + s_0 \log 5}$ for $\delta \in (0,1)$ and $\eps \in (0,1-C')$ where $c_0$ is a positive constant. Then
with probability at least $1 - \delta$,
\bal\label{eq:error-bound-noncovnex-sparse-eigenvalue}
\ltwonorm{\tilde \bbeta^* - \bbeta^*}
\le \frac{\eps \sqrt{\rho_{\cL,+}(s_0)}}{(1-\eps)\rho_{\cL,-}(s_0) - \zeta_{-}  }  \norm{\bbeta^*}{\bX}.
\eal
\end{corollary}
The following assumption, Assumption~\ref{assumption:nonconvex-sketch-sparse-recovery}, is necessary to achieve the minimax parameter estimation error by sketching, and the subsequent Remark~\ref{remark:nonconvex-sketch-sparse-recovery-assumption} explains that Assumption~\ref{assumption:nonconvex-sketch-sparse-recovery} is mild. That is, if Assumption~\ref{assumption:nonconvex-sparse-recovery}
holds, then Assumption~\ref{assumption:nonconvex-sketch-sparse-recovery} also holds under mild conditions.
\begin{assumption}
(Assumption in \cite{Wang2014-sparse-nonconvex} for sparse signal estimation)
\label{assumption:nonconvex-sketch-sparse-recovery}
Let $\tilde s,C'$ be the parameters specified in
Assumption~\ref{assumption:nonconvex-sparse-recovery} such that Assumption~\ref{assumption:nonconvex-sparse-recovery}
holds. Then it is assumed that $\rho_{\tilde \cL,+}({\bar s}+2\tilde s) < \infty, \rho_{\tilde \cL,-}({\bar s}+2\tilde s) > 0$ are two absolute constants. In addition, $\zeta_{-}$ satisfies
$\zeta_{-} \le C' \rho_{\tilde \cL,-}({\bar s}+2\tilde s)$, and $\tilde s > \tilde C \bar s$ where $\tilde C =144 \tilde \kappa^2 + 250 \tilde \kappa$ with $\tilde \kappa \defeq (\rho_{\tilde \cL,+}({\bar s}+2 \tilde s) - \zeta_+)/(\rho_{\tilde \cL,-}({\bar s}+2\tilde s) - \zeta_{-})$.
\end{assumption}
\begin{remark}
[Assumption~\ref{assumption:nonconvex-sketch-sparse-recovery} is mild]
\label{remark:nonconvex-sketch-sparse-recovery-assumption}
We provide theoretical justification that Assumption~\ref{assumption:nonconvex-sketch-sparse-recovery} is mild.
The following theorem, Theorem~\ref{theorem:satisfying-assumption}, shows that if the standard Assumption~\ref{assumption:nonconvex-sparse-recovery} holds, then
Assumption~\ref{assumption:nonconvex-sketch-sparse-recovery} also holds with high probability under very mild conditions: either $\eps$ is set to the order of
$\sqrt{\log d/n}$ with sufficiently large $n$, or the Restricted
Isometry Property (RIP) \cite{CandesTao2005-decoding-linear-programming} holds. It is well known that RIP holds for various choices of the design matrix $\bX$, and \cite{Wang2014-sparse-nonconvex} also uses RIP to justify that the standard Assumption~\ref{assumption:nonconvex-sparse-recovery} is weaker than RIP.
\end{remark}
\begin{figure*}[!htbp]
\centering
\includegraphics[width=0.7\textwidth]{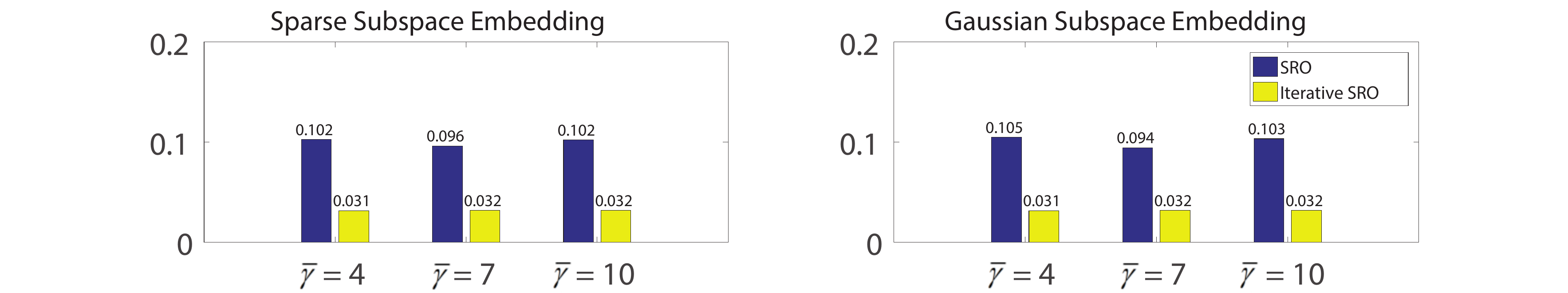}
\caption{Approximation error of Iterative SRO vs. SRO for GLasso with respect to different sketch size $\tilde n$. Iterative SRO and SRO are equipped with either sparse subspace embedding (left) or Gaussian Subspace Embedding (right).}
\label{fig:SRO-glasso-mean}
\end{figure*}
\begin{theorem}\label{theorem:satisfying-assumption}
Suppose Assumption~\ref{assumption:nonconvex-sparse-recovery}
holds with $\tilde s$ and $C'$ specified in
Assumption~\ref{assumption:nonconvex-sparse-recovery}, and let
$s_0 = {\bar s}+2\tilde s$.
Let $0 < \eps,\delta < 1$, $\bP \in \RR^{\tilde n \times n}$ be a Gaussian subspace embedding defined in Definition~\ref{def:gaussian-subspace-embedding}, and
$\tilde n \ge c_0\eps^{-2}\pth{\log(2/\delta) + s_0\log d + s_0 \log 5 +1/d^{s_0-1} } $ where $c_0$ is a positive constant.
If $\rho_{\cL,-}(s_0) > \eps \sqrt{s_0}$, $\zeta_{-} \le C' \pth{\rho_{\cL,-}(s_0)- \eps \sqrt{s_0}}$ and $\pth{144 \kappa'^2 + 250 \kappa'} {\bar s} < \tilde s$ with $\kappa' = (\rho_{\cL,+}(s_0) + \eps \sqrt{s_0}- \zeta_+)/(\rho_{\cL,-}(s_0) -\eps \sqrt{s_0} - \zeta_{-})$, then with probability at least $1-\delta$, Assumption~\ref{assumption:nonconvex-sketch-sparse-recovery} holds.
In particular, Assumption~\ref{assumption:nonconvex-sketch-sparse-recovery} holds with probability at least $1-\delta$ if any one of the following two conditions holds:
\begin{itemize}[leftmargin=18pt]
\setlength\itemsep{0.1em}
\item[(a)] $\log d/n \overset{n \to \infty}{\longrightarrow} 0$, $\eps = C_1 \sqrt{\log d/n}$ with $C_1$ being a positive constant and $n$ sufficiently large;
\item[(b)] There exists $s' \ge s_0$ such that $\textup{RIP}(\delta,s')$ holds for $\delta \in (0,1)$,
$\zeta_{+} = 0$,
$\zeta_{-} = C_2\rho_{\cL,-}(s_0)$, $\eps\sqrt{s_0} \le
C_3\rho_{\cL,-}(s_0)$, $\tilde s > \pth{144 \kappa_0^2 + 250 \kappa_0} \bar s$ with $\kappa_0 = \pth{(1+C_3)(1+\delta)}/\pth{(1-C_2-C_3)(1-\delta)}$. Here the positive constants $C_2,C_3$ satisfy $C_2+C_3 < 1$ and $C_2 \le C'(1-C_3)$.
$\textup{RIP}(\delta,s)$ for $\delta \in (0,1)$ and $s \in \NN$ is the Restricted
Isometry Property (RIP) \cite{CandesTao2005-decoding-linear-programming} under which
$1-\delta \le \rho_{\cL,-}(s) \le \rho_{\cL,+}(s) \le 1+\delta$ holds.
\end{itemize}
\end{theorem}
It is shown in Section~\ref{sec:nonconvex-penalty} that $\zeta_{+} = 0$ and $\zeta_{-} = C_2\rho_{\cL,-}(s_0)$ can be easily achieved by setting the hyperparameter of MCP when MCP is used as the nonconvex regularizer $h_{\lambda}$.  We have the following sharp bound for the parameter estimation error with sparse nonconvex learning by sketching in Theorem~\ref{theorem:sketching-nonconvex-minimax}. We note that the approximate path following method described in \cite[Algorithm 1]{Wang2014-sparse-nonconvex} is used to solve the original problem (\ref{eq:optimization-general}) and the sketched problem (\ref{eq:optimization-general-rp}) to obtain $\tilde \bbeta^*$ and $\tilde \bbeta^*$ such that
$\bbeta^*$ is an critical point of problem (\ref{eq:optimization-general}) and $\tilde \bbeta^*$
is an critical point of (\ref{eq:optimization-general-rp}). We use $\lambda = \Theta\pth{\sqrt{\bar s \log d/n}}$ for both (\ref{eq:optimization-general}) and (\ref{eq:optimization-general-rp}) at the final stage of the path following method \cite[Algorithm 1]{Wang2014-sparse-nonconvex}.
\begin{theorem}
\label{theorem:sketching-nonconvex-minimax}
Let $\delta \in (0,1)$ and $\bP \in \RR^{\tilde n \times n}$ be a Gaussian subspace embedding defined in Definition~\ref{def:gaussian-subspace-embedding}, and $\eps = \min\set{C_1 \sqrt{\log d/n}, \eps_0}$ with $C_1,\eps_0$ being positive constants and $\eps \in (0,(1-C')/2)$. Suppose
Assumption~\ref{assumption:nonconvex-sparse-recovery} and
Assumption~\ref{assumption:nonconvex-sketch-sparse-recovery} hold, $d \ge 5$, and let $s_0 = {\bar s}+2\tilde s$. Then under the conditions of Theorem~\ref{theorem::error-bound-general}, with probability at least $1- 4/d$,
\bal\label{eq:sketching-nonconvex-minimax}
\ltwonorm{\tilde \bbeta^* - \bar \bbeta}
\le \frac{C_1\ltwonorm{\bar \bbeta} \sqrt{(1 + \tilde s/\bar s)\rho_{\cL,+}(s_0)}}{(1+C')/2 \cdot\rho_{\cL,-}(s_0) - \zeta_{-}  }  \sqrt{\frac{\bar s\log d}n} + \frac{22\pth{C_1 \sqrt{{\bar s}} \ltwonorm{\bar \bbeta}  +  2 \sigma}}{\rho_{\cL,-}(s_0)-\zeta_{-}}  \sqrt{\frac{\bar s \log d}{n}}.
\eal
Moreover, let $C_1 = \sqrt{c_0792 s_0/789} C_3$ with $C_3 > 1$ being a positive constant. Then (\ref{eq:sketching-nonconvex-minimax}) holds
with $\tilde n \ge n /C_3^2$ for $n \ge \Theta(1)$.
\end{theorem}
It is noted that we can choose $\tilde s = \Theta(\bar s)$. When $\ltwonorm{\bar \bbeta}$ is a constant, we have
the parameter estimation error $\ltwonorm{\tilde \bbeta^* - \bar \bbeta} \le \cO\pth{\sqrt{\bar s \log d/{n}}}$, which is the minimax estimation error according to \cite{Wang2014-sparse-nonconvex}. Moreover, with
$C_1 = 4\sqrt{792 s_0/789} C_3$, we can enjoy a small sketch size $\tilde n = n /C_3^2$ with a potentially large $C_3 > 1$, and this is at the expense of having a large constant factor $C_1$ in the parameter estimation error $\cO\pth{\sqrt{\bar s \log d/{n}}}$.

\section{Time Complexity}
We compare the time complexity of solving the original problem (\ref{eq:optimization-general}) to that of solving the sketched problem (\ref{eq:optimization-general-rp}) with Iterative SRO, which is deferred to Section~\ref{sec::time-complexity} of the supplementary.

\section{Experimental Results}\label{sec::experiments}

We provide empirical results in this section to justify the effectiveness of the proposed SRO and Iterative SRO.

\subsection{Generalized Lasso}
\label{sec:GLasso}
We study the performance of Iterative SRO for Generalized Lasso (GLasso) \cite{tibshirani2011-generalized-lasso} in this subsection. The optimization problem of an instance of GLasso studied here is $\bbeta^* = \argmin_{\bbeta \in \RR^d} \frac{1}{2} \ltwonorm{\by - \bX \bbeta}^2 + \lambda \sum\limits_{i=1}^{d-1} |\bbeta_i - \bbeta_{i+1}|$, which is solved by Fast Iterative Shrinkage-Thresholding Algorithm (FISTA) \cite{Beck2009-fast-ista}, an accelerated version of PGD. $\bar \bX = \sqrt{n} \bX \in \RR^{n \times d}$ have i.i.d. standard Gaussian entries with $n = 80000$ and $d = 600$, and all the elements of $\bar \by = \sqrt{n} \by$ are also i.i.d. Gaussian samples. Figure~\ref{fig:SRO-glasso-mean} illustrates the approximation error of SRO and Iterative SRO, which are $\frac{\norm{\bbeta^{(1)} -\bbeta^*}{\bX}^2}{n}$ and $\frac{\norm{\bbeta^{(N)} -\bbeta^*}{\bX}^2}{n}$ respectively, for different choices of sketch size $\tilde n$ with $\tilde n = \bar \gamma d$. We set $N=10$ and employ either sparse subspace embedding or Gaussian subspace embedding. The average approximation errors are reported over $100$ trials of data sampling for each $\bar \gamma$. It can be observed that Iterative SRO significantly reduces the approximation error and its approximation error is roughly $\frac{1}{3}$ of that of SRO, demonstrating the effectiveness of Iterative SRO. In our experiment, due to the significant reduction in the sample size, for example, $\frac{{\bar \gamma} d}{n} = 0.03$ when $\bar \gamma = 4$, the running time of Iterative SRO is always less than half of that required to solve the original problem with small $\bar \gamma$.

We also report the running time of GLasso with sparse subspace embedding in Table~\ref{table:SRO-glasso-time}. We use $M=10000$ iterations for FISTA, and the running time is reported for $\bar \gamma = 4$ on a CPU of Intel i5-11300H. Iterative SRO-IHS is the ``IHS'' version of Iterative SRO where a new sketch matrix $\bP$ is sampled and the sketched data $\tbX = \bP \bX$ is computed at each iteration. We observed that both Iterative SRO-IHS and Iterative SRO achieve the same approximation error, while Iterative SRO is faster than Iterative SRO-IHS because the former only samples the linear transformation $\bP$ once and computes the sketched matrix $\tbX$ once.
\begin{table}[!hbt]
\centering
\caption{Running time (in seconds) of SRO, Iterative SRO with $\bar \gamma=4$ and Iterative SRO-IHS for GLasso. The number in the bracket is the approximation error. }
\begin{tabular}{|c|c|c|}
\hline
SRO       &Iterative SRO      &Iterative SRO-IHS   \\ \hline
$11.57s$   &$4.91s (0.031)$             &$5.32s (0.031)$      \\ \hline

\end{tabular}
\label{table:SRO-glasso-time}
\end{table}
\subsection{Additional Experiments}
\label{sec:additional-experiments-main}
We defer more experimental results to the Section~\ref{sec::complete-experiments} of the supplementary. In particular, experimental results for ridge regression and sparse signal estimation by Lasso are in Section~\ref{sec::ridge-regression} and Section~\ref{sec::lasso} respectively, and
more details about GLasso are in Section~\ref{sec::more-glasso}. We further apply SRO to subspace clustering using Lasso in Section~\ref{sec:LSC}.


\section{Conclusion}
We present Sketching for Regularized Optimization (SRO) which efficiently solves general regularized optimization problems with convex or nonconvex regularization by sketching. We further propose Iterative SRO to reduce the approximation error of SRO geometrically, and provide a unified theoretical framework under which the minimax rates for sparse signal estimation are obtained for both convex and nonconvex sparse learning problems. Experimental results evidence that Iterative SRO can effectively and efficiently approximate the optimization result of the original problem.


\bibliographystyle{IEEEtran}
\bibliography{ref}

\section*{Checklist}

The checklist follows the references. For each question, choose your answer from the three possible options: Yes, No, Not Applicable.  You are encouraged to include a justification to your answer, either by referencing the appropriate section of your paper or providing a brief inline description (1-2 sentences).
Please do not modify the questions.  Note that the Checklist section does not count towards the page limit. Not including the checklist in the first submission won't result in desk rejection, although in such case we will ask you to upload it during the author response period and include it in camera ready (if accepted).

\textbf{In your paper, please delete this instructions block and only keep the Checklist section heading above along with the questions/answers below.}

 \begin{enumerate}

 \item For all models and algorithms presented, check if you include:
 \begin{enumerate}
   \item A clear description of the mathematical setting, assumptions, algorithm, and/or model. [Yes]
   \item An analysis of the properties and complexity (time, space, sample size) of any algorithm. [Yes]
   \item (Optional) Anonymized source code, with specification of all dependencies, including external libraries. [Yes]
 \end{enumerate}

 \item For any theoretical claim, check if you include:
 \begin{enumerate}
   \item Statements of the full set of assumptions of all theoretical results. [Yes]
   \item Complete proofs of all theoretical results. [Yes]
   \item Clear explanations of any assumptions. [Yes]
 \end{enumerate}

 \item For all figures and tables that present empirical results, check if you include:
 \begin{enumerate}
   \item The code, data, and instructions needed to reproduce the main experimental results (either in the supplemental material or as a URL). [Yes]
   \item All the training details (e.g., data splits, hyperparameters, how they were chosen). [Yes]
         \item A clear definition of the specific measure or statistics and error bars (e.g., with respect to the random seed after running experiments multiple times). [Yes]
         \item A description of the computing infrastructure used. (e.g., type of GPUs, internal cluster, or cloud provider). [Yes]
 \end{enumerate}

 \item If you are using existing assets (e.g., code, data, models) or curating/releasing new assets, check if you include:
 \begin{enumerate}
   \item Citations of the creator if your work uses existing assets. [Not Applicable]
   \item The license information of the assets, if applicable. [Not Applicable]
   \item New assets either in the supplemental material or as a URL, if applicable. [Not Applicable]
   \item Information about consent from data providers/curators. [Not Applicable]
   \item Discussion of sensible content if applicable, e.g., personally identifiable information or offensive content. [Not Applicable]
 \end{enumerate}

 \item If you used crowdsourcing or conducted research with human subjects, check if you include:
 \begin{enumerate}
   \item The full text of instructions given to participants and screenshots. [Not Applicable]
   \item Descriptions of potential participant risks, with links to Institutional Review Board (IRB) approvals if applicable. [Not Applicable]
   \item The estimated hourly wage paid to participants and the total amount spent on participant compensation. [Not Applicable]
 \end{enumerate}

 \end{enumerate}

\newpage
\appendix

\section{Nonconvex Penalty}
\label{sec:nonconvex-penalty}
We introduce more details about the nonconvex regularizer $h_{\lambda}$
for sparse nonconvex learning in Subsection~\ref{sec:sparse-nonconvex-learning}.
$h_{\lambda} = \sum_{j=1}^d p_{\lambda}(\beta_j)$, and $p_{\lambda}$ can be either smoothly clipped absolute deviation (SCAD) \cite{Fan2001-SCAD} or minimax concave penalty (MCP) \cite{Zhang2010-MCP}.
When $p_{\lambda}$ is SCAD, we have
\bals
p_{\lambda}(\beta_j) = \lambda \int_0^{\abth{\beta_j}}
\pth{\indict{z \le \lambda} + \frac{(a\lambda-z)_+}{(a-1)\lambda} \indict{z > \lambda} } \diff z, \quad a > 2.
\eals
When $p_{\lambda}$ is SCAD, we have
\bals
p_{\lambda}(\beta_j) = \lambda \int_0^{\abth{\beta_j}}
\pth{1-\frac{z}{\lambda b}}_+ \diff z, \quad b > 0.
\eals

\section{Time Complexity}\label{sec::time-complexity}

We compare the time complexity of solving the original problem (\ref{eq:optimization-general}) to that of solving the sketched problem (\ref{eq:optimization-general-rp}) with Iterative SRO. We employ Proximal Gradient Descent (PGD) or Gradient Descent (GD) in our analysis, which are widely used in the machine learning and optimization literature. If $\bP$ is a Gaussian subspace embedding in Definition~\ref{def:gaussian-subspace-embedding}, it takes $\cO(\tilde n nd)$ operations to compute the sketched matrix $\tbX = \bP \bX$ and then form the sketched problem (\ref{eq:optimization-general-rp}). Let $C(\tilde n,d)$ be the time complexity of solving the sketched problem (\ref{eq:optimization-general-rp}), and suppose iterative sketching is performed for $N$ iterations, then the overall time complexity of Iterative SRO in Algorithm~\ref{alg:iterative-sketch-SRO} is $\cO\left( \tilde n nd + N C(\tilde n,d) \right) $. If $\bP$ is a sparse subspace embedding in Definition~\ref{def:sparse-subspace-embedding}, then it only takes $\cO\pth{{\textsf {nnz}(\bX)}}$ operations to compute the sketched matrix $\tbX$. In this case, the overall time complexity of Iterative SRO is $\cO \left({\textsf {nnz}(\bX)} + NC(\tilde n,d) \right)$. Suppose PGD, such as that analyzed in \cite{BoltePAL2014}, or GD, is used to solve problem (\ref{eq:optimization-general}) and (\ref{eq:optimization-general-rp}) with maximum number of iterations being $M$. Then $C(\tilde n,d) = \cO \pth{ M \tilde n d}$. If a sparse subspace embedding is used for sketching, then the overall time complexity of Iterative SRO is $\cO\pth{{\textsf {nnz}(\bX)} + N M \tilde n d }$. In contrast, because IHS \cite{Pilanci2016-IHS} needs to sample an independent sketch matrix at each iteration, the time complexity of IHS using the fast Johnson-Lindenstrauss sketches (that is, the fast Hadamard transform) is $\cO \left(N n d \log{\tilde n} + N M \tilde n d \right)$ which is higher than that of Iterative SRO with sparse subspace embedding. Noting that $N \le \log n$ \cite{Pilanci2016-IHS} and in many practical cases $N$ is bounded by a constant, and $\tilde n \ll n$, the complexity of Iterative SRO, $\cO\pth{{\textsf {nnz}(\bX)} + N M \tilde n d }$, is much lower than that of solving the original problem (\ref{eq:optimization-general}) with the complexity of $\cO(Mnd)$.

\section{Complete Experimental Results}
\label{sec::complete-experiments}
We demonstrate complete experimental results of SRO and Iterative SRO in this section for three instances of the general optimization problem (\ref{eq:optimization-general}), which include ridge regression where $h(\bbeta) = \ltwonorm{\bbeta}^2$, Generalized Lasso where $h(\bbeta) = \sum\limits_{i=1}^{d-1} |\bbeta_i - \bbeta_{i+1}|$, and Lasso where $h(\bbeta) = \lonenorm{\bbeta}$.  We further show the application of SRO in subspace clustering.

\begin{figure*}[!htb]
\hspace{-0.7in}
\includegraphics[width=8in]{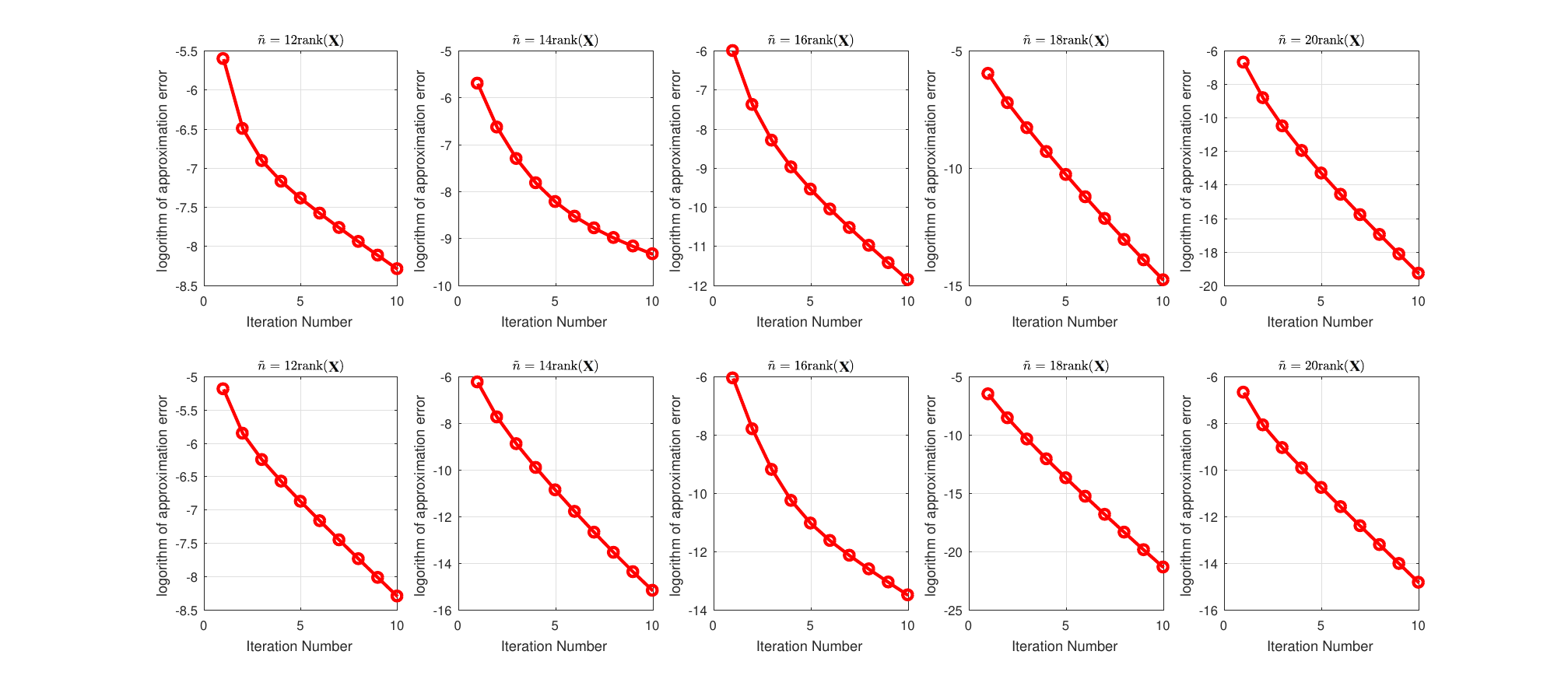}
\caption{Approximation Error of Iterative SRO with respect to different sketch size $\tilde n$ for ridge regression. The first row corresponds to sparse subspace embedding defined in Definition~\ref{def:sparse-subspace-embedding} and the second row is produced by Gaussian subspace embedding defined in Definition~\ref{def:gaussian-subspace-embedding}. }
\label{fig:SRO-rr-mean}
\end{figure*}

\begin{figure*}[!hbt]
\hspace{-0.7in}
\includegraphics[width=8in]{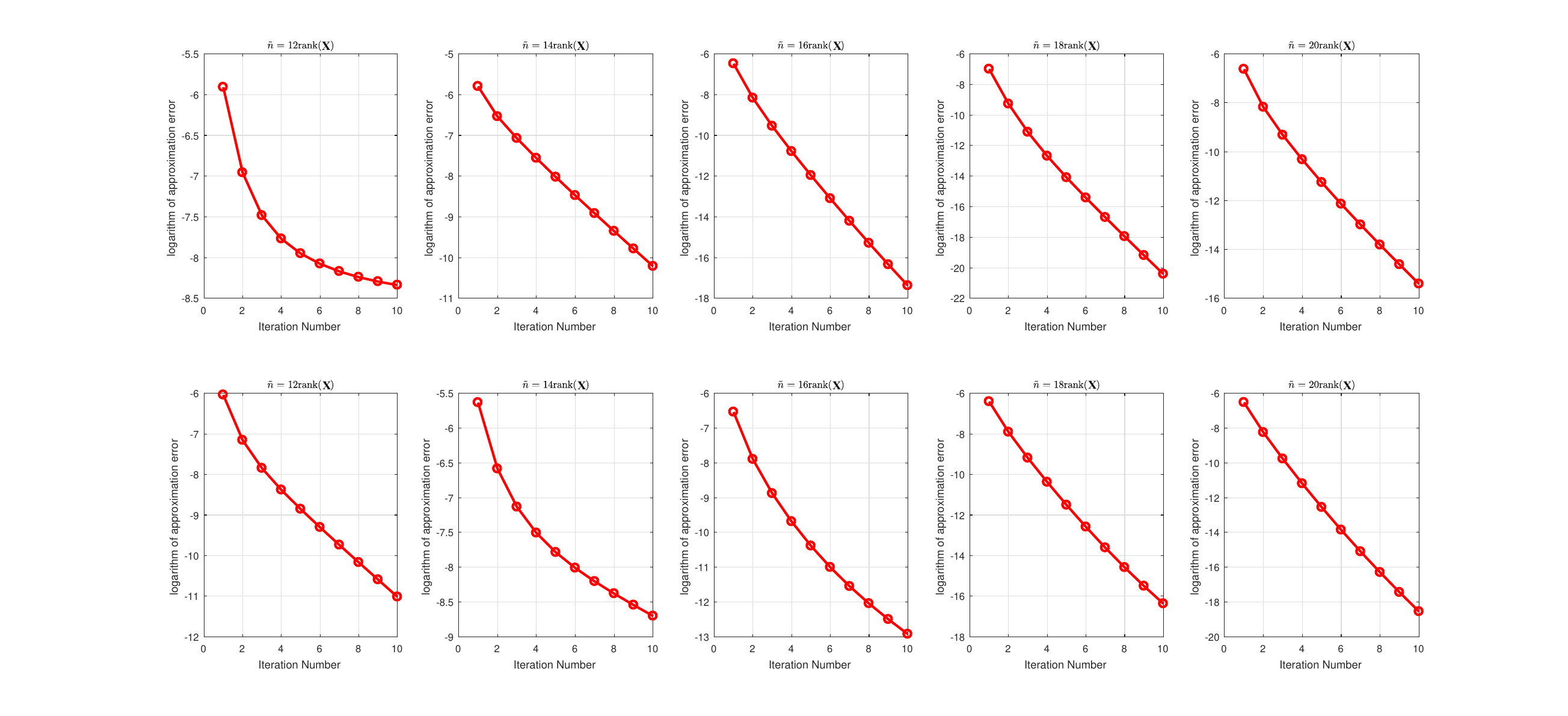}
   \caption{Logarithm of approximation error of Iterative SRO with respect to different sketch size $\tilde n$ for ridge regression. The first row corresponds to sparse subspace embedding and the second row is produced by Gaussian subspace embedding. }
\label{fig:SRO-rr-mean-supp}
\end{figure*}

\begin{figure*}[!hbt]
\hspace{-0.7in}
\includegraphics[width=8in]{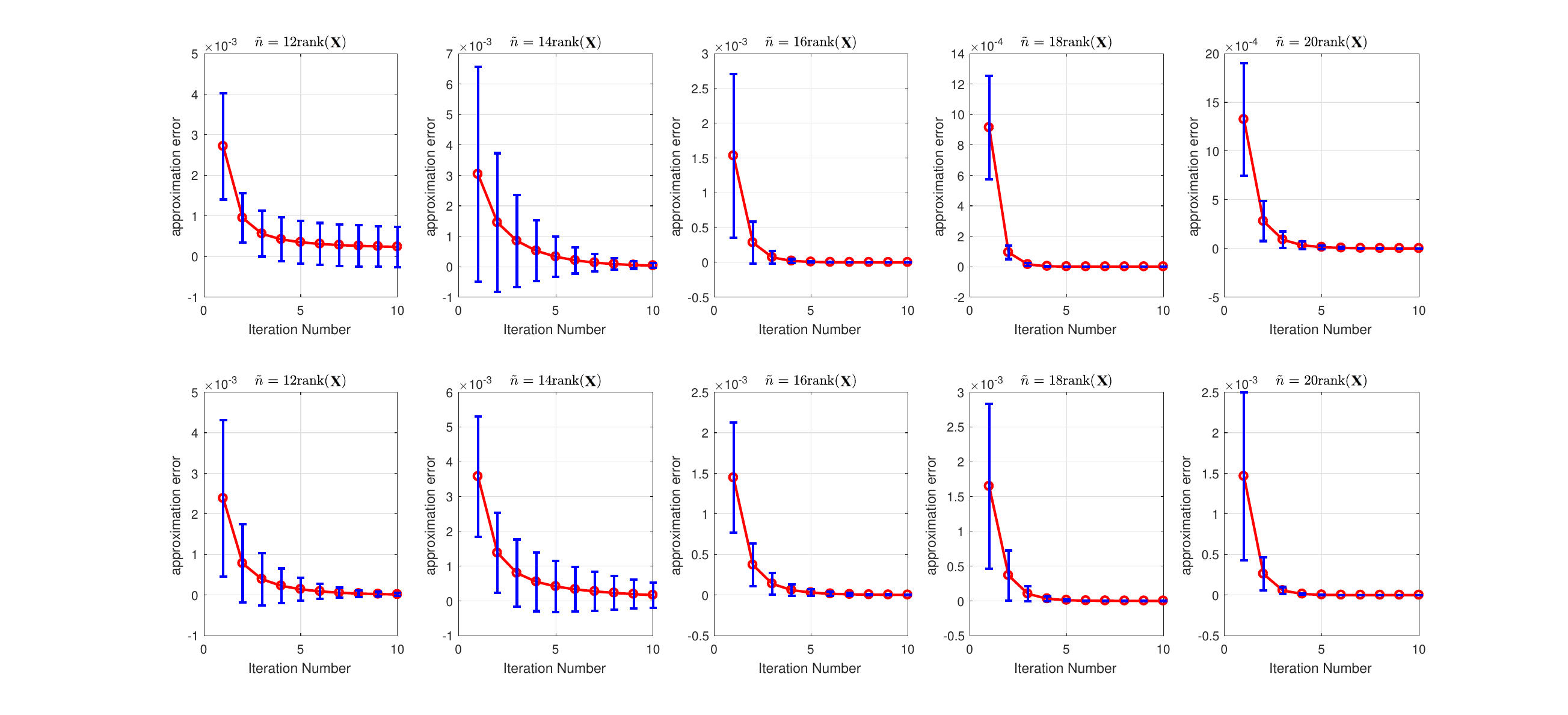}
   \caption{Approximation error of Iterative SRO with respect to different sketch size $\tilde n$ for ridge regression. The first row corresponds to sparse subspace embedding and the second row is produced by Gaussian subspace embedding.}
\label{fig:SRO-rr-std-supp}
\end{figure*}

\subsection{Ridge Regression of Larger Scale}
\label{sec::ridge-regression}

We employ Iterative SRO to approximate solution to Ridge Regression in this subsection, whose optimization problem is $\bbeta^* = \argmin_{\bbeta \in \RR^d} \frac{1}{2} \ltwonorm{\by - \bX \bbeta}^2 + \lambda \ltwonorm{\bbeta}$. We assume a linear model $\bar \by = \bar \bX \bar \bbeta + \bw$ where $\bX = \bar \bX/\sqrt{n}$, $\by = \bar \by/\sqrt{n}$, $\bw \sim \cN(\bzero, \bI_n)$ is the Gaussian noise with unit variance. The unknown regression vector $\bar \bbeta$ is  sampled according to $\cN(\bzero, \bI_d)$. We randomly sample $\bar \bX \in \RR^{n \times d}$ of rank $r=\frac{n}{100}$ with $n=5000$ and $d = 10000$. Let $\bar \bX = \bU \bSigma \bV^{\top}$ be the Singular Value Decomposition of $\bar \bX$ where $\bSigma \in \RR^{r \times r}$ is a diagonal matrix whose diagonal elements are the singular values of $\bar \bX$. $\bU \in \RR^{n \times r}$ is sampled from the uniform distribution over the Stiefel manifold $\mathbbm{V}_r(\RR^n)$, $\bV \in \RR^{d \times r}$ is sampled from the uniform distribution over the Stiefel manifold $\mathbbm{V}_r(\RR^d)$, the diagonal elements of $\bSigma$ are i.i.d. standard Gaussian samples. We set $\lambda = \sqrt{\log d/n}$. Figure~\ref{fig:SRO-rr-mean} illustrates the logarithm of approximation error $\frac{\norm{\bbeta^{(i)} -\bbeta^*}{\bX}^2}{n}$ with respect to the iteration number $i$ of Iterative SRO for different choices of sketch size $\tilde n$, and the maximum iteration number of Iterative SRO is set to $N = 10$. We let $\tilde n = \bar \gamma {\rm rank}(\bX)$ where $\bar \gamma$ ranges over $\set{12,14,16,18,20}$, and sample $\bar \bX,\bar \by$ and $\bw$ $100$ times for each $\bar \gamma$. The average approximation errors are illustrated in Figure~\ref{fig:SRO-rr-mean}. It can be observed from Figure~\ref{fig:SRO-rr-mean} that the convergence rate of approximation error drops geometrically, or its logarithm drops linearly, evidencing our Theorem~\ref{theorem::iterative-sketch} for Iterative SRO. Moreover, as suggested by Theorem~\ref{theorem::error-bound-general}, larger $\tilde n$ leads to smaller approximation error.

We present more experimental results for ridge regression with larger-scale data where $n=10000$ and $d = 100000$, and other settings remain the same. We let $\tilde n = \bar \gamma {\rm rank}(\bar \bX)$ where $\bar \gamma$ ranges over $\{12,14,16,18,20\}$, and sample $\bar \bX,\bar \bbeta$ and $\bw$ $100$ times for each $\bar \gamma$. Figure~\ref{fig:SRO-rr-mean-supp} illustrates the logarithm of approximation error, which is $\log{\frac{\norm{\bbeta^{(i)} -\bbeta^*}{\bX}^2}{n}}$, with respect to the iteration number $i$ of Iterative SRO for different choices of sketch size $\tilde n$, and the maximum iteration number of Iterative SRO is set to $N = 10$. It can be observed that the logarithm of approximation error drops linearly with respect to the iteration number in most cases, evidencing our theory that the approximation error of Iterative SRO drops geometrically in the iteration number.

Figure~\ref{fig:SRO-rr-std-supp} illustrates the approximation error of Iterative SRO in red curve for ridge regression, which is $\frac{\norm{\bbeta^{(i)} -\bbeta^*}{\bX}^2}{n}$. The blue bar represents standard deviation caused by the random data sampling and the random sketching. A single projection matrix $\bP$ is sampled for each sampled data, and this projection matrix is used throughout all the iterations of Iterative SRO, in contrast with Iterative Hessian Sketch (IHS) \cite{Pilanci2016-IHS} which samples a projection matrix matrix for each iteration of the iterative sketch procedure.

\subsection{Generalized Lasso}
\label{sec::more-glasso}
In this subsection, we add more details to Section~\ref{sec:GLasso} of the paper for Generalized Lasso (GLasso) \cite{tibshirani2011-generalized-lasso}. The optimization problem of GLasso studied here is $\bbeta^* = \argmin_{\bbeta \in \RR^d} \frac{1}{2} \ltwonorm{\by - \bX \bbeta}^2 + \lambda \sum\limits_{i=1}^{d-1} |\bbeta_i - \bbeta_{i+1}|$, which is solved by Fast Iterative Shrinkage-Thresholding Algorithm (FISTA) \cite{Beck2009-fast-ista}. We construct $\bD \in \RR^{(d-1) \times d}$ by setting $\bD_{i,i} =-1$, $\bD_{i,i+1} =1$ for all $i \in [d-1]$, then $\sum\limits_{i=1}^{d-1} |\bbeta_i - \bbeta_{i+1}| = \lonenorm{\bD \bbeta}$. Let $\bD^{\rm ext} = \left[ {\begin{array}{*{20}{c}}
{\bD}\\
{0,\ldots,1}
\end{array}} \right]$ with $\bD_{dd}^{\rm ext} = 1$. Denote $\bu = \bD^{\rm ext}  \bbeta$, then $\bbeta = \left(\bD^{\rm ext}\right)^{-1} \bbeta$ because $\bD^{\rm ext}$ is nonsingular. The instance of GLasso considered above is then rewritten as $\bbeta^* = \argmin_{\bbeta \in \RR^d} \frac{1}{2n} \ltwonorm{\by - \bX  \left(\bD^{\rm ext}\right)^{-1} \bu}^2 + \lambda \sum\limits_{i=1}^{d-1} \abth{\bu_i}$ which can be solved by FISTA.

\subsection{Signal Recovery by Lasso}
\label{sec::lasso}
We present experimental results for signal recovery/approximation by Lasso in this subsection.  In this experiment we assume a linear model $\bar \by = \bar \bX \bar \bbeta + \bw$ where $\bX = \bar \bX/\sqrt{n}$, $\by = \bar \by/\sqrt{n}$, $\bw \sim \cN(\bzero, \bI_n)$ is the Gaussian noise with unit variance. The optimization problem of Lasso considered here is $\bbeta^* = \argmin_{\bbeta \in \RR^d} \frac{1}{2} \ltwonorm{\by - \bX \bbeta}^2 + \lambda \lonenorm{\bbeta}$. We set $\lambda = 0.1 \cdot \sqrt{\bar s \log d/n}$, sparsity $\bar s = \floor*{3 \log{d}}$, and choose the unknown regression vector $\bar \bbeta$ with its support uniformly sampled with entries $\pm \frac{1}{\sqrt s}$ with equal probability. We randomly sample $\bar \bX \in \RR^{n \times d}$ of rank not greater than $r = \frac{n}{100}$ with $n=5000$ and $d = 100000$ using (\ref{eq:low-rank-RIP-X}) so that $\bar \bX$ is a low-rank matrix satisfying RIP. That is, $\bX = \bU \bU^{\top} \bOmega$ with $\bU \in \RR^{n \times r}$ and $\bOmega \in \RR^{n \times d}$. $\bU$ is sampled from the uniform distribution over the Stiefel manifold $\mathbbm{V}_r(\RR^n)$, the elements of $\bOmega$ are i.i.d. Gaussian random variables
with $\bOmega_{ij} \sim \cN(0,1/r)$ for $i \in [n], j \in [d]$.
We let $\tilde n = \bar \gamma {\rm rank}(\bar \bX)$ where $\bar \gamma$ ranges over $\set{12,14,16,18}$, and sample $\bar \bX,\bar \bbeta$ and $\bw$ $100$ times for each $\gamma$. Figure~\ref{fig:SRO-lasso-mean} illustrates the approximation error of SRO and Iterative SRO, i.e. $\frac{\norm{\bbeta^{(1)} -\bbeta^*}{\bX}^2}{n}$ and $\frac{\norm{\bbeta^{(N)} -\bbeta^*}{\bX}^2}{n}$ respectively, for different choices of sketch size $\tilde n$ with different $\bar \gamma$. It can be observed that Iterative SRO significantly and constantly reduces approximation error of SRO.

Table~\ref{table:SRO-lasso-appro-err} shows the error of approximation to the true unknown regression vector $\bar \bbeta$ by SRO, Iterative SRO and the solution $\bbeta^*$ to the original problem (\ref{eq:optimization-general}).
The error of approximation to $\bar \bbeta$ is the $\ell^2$-distance to $\bar \bbeta$, for example, the error of approximation to $\bar \bbeta$ for $\bbeta^*$ is $\ltwonorm{\bbeta^*-\bar \bbeta}$.
 Standard deviation is caused by random data sampling and random sketching. It can be seen that Iterative SRO approximates the true regression vector much better than SRO with the sketch size $\tilde n$ being a fraction of $n$, especially with Gaussian subspace embedding. More importantly, Iterative SRO and $\bbeta^*$ have very close error of approximation to $\bar \bbeta$,
justifying our theoretical analysis. This is because that our theoretical prediction in Theorem~\ref{theorem:sketching-convex-minimax} for the error of approximation to  $\bar \bbeta$ by Iterative SRO is $\Theta\pth{\sqrt{\bar s \log d/n}}$, which is the same order as the error by $\bbeta^*$.

\begin{figure*}[!hbt]
\begin{center}
\includegraphics[width=6in]{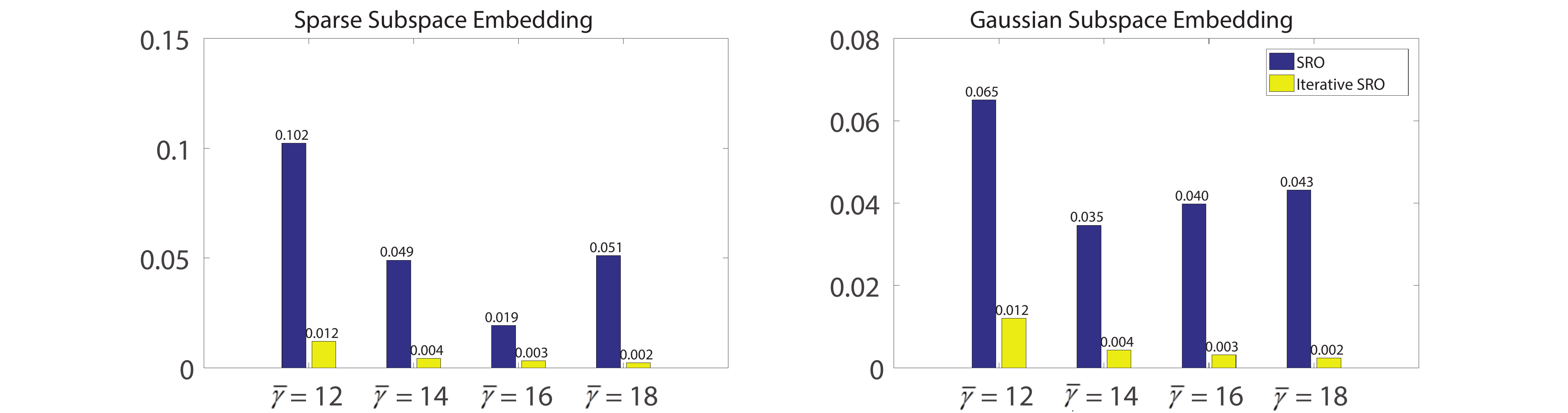}
\end{center}
\caption{Approximation error of Iterative SRO and SRO for Lasso with respect to different sketch size $\tilde n = \gamma {\rm rank}(\bX)$. Iterative SRO and SRO are equipped with either sparse subspace embedding (left) or Gaussian Subspace Embedding (right).}
\label{fig:SRO-lasso-mean}
\end{figure*}

\begin{table}[!hbt]
\centering
\caption{\small Approximation error to the true unknown regression vector $\bar \bbeta$ by SRO, Iterative SRO and the solution $\bbeta^*$ to the original problem (\ref{eq:optimization-general}) for Lasso.}
\resizebox{1\linewidth}{!}{
\begin{tabular}{|c|c|c|c|c|c|c|}
  \hline

  &\backslashbox{Error}{$\bar \gamma$}                       &$12$                    &$14$                   &$16$                   &$18$              \\\hline

  \multirow{3}{*}{Sparse Subspace Embedding}
  &SRO                   &$0.115 \pm 0.149$       &$0.063 \pm 0.023$      &$0.059 \pm 0.010$      &$0.025 \pm 0.054$            \\ \cline{2-6}
  &Iteratie SRO          &$0.010 \pm 0.017$       &$0.009 \pm 0.009$      &$0.004 \pm 0.002$      &$0.008 \pm 0.005$            \\ \cline{2-6}
  &$\bbeta^*$       &$0.008 \pm 0.002$       &$0.007 \pm 0.006$      &$0.003 \pm 0.001$      &$0.006 \pm 0.004$            \\ \hline

  \multirow{3}{*}{Gaussian Subspace Embedding}
  &SRO                   &$0.067 \pm 0.092$       &$0.043 \pm 0.066$      &$0.044 \pm 0.038$      &$0.046 \pm 0.030$            \\ \cline{2-6}
  &Iteratie SRO          &$0.008 \pm 0.003$       &$0.008 \pm 0.007$      &$0.004 \pm 0.002$      &$0.003 \pm 0.004$            \\ \cline{2-6}
  &$\bbeta^*$       &$0.007 \pm 0.001$       &$0.006 \pm 0.003$      &$0.003 \pm 0.001$      &$0.002 \pm 0.002$            \\ \hline
\end{tabular}
}
\label{table:SRO-lasso-appro-err}
\end{table}

\subsection{Application in Lasso Subspace Clustering}
\label{sec:LSC}
We demonstrate application of SRO in subspace clustering in this subsection. Given a data matrix $\bar \bX \in \RR^{n \times d}$ comprised of $d$ data points in $\RR^n$ which lie in a union of subspaces in $\RR^n$, classical subspace clustering methods using sparse codes, such as Noisy Sparse Subspace Clustering (Noisy SSC) \cite{WangX13}, recovers the subspace structure by solving the Lasso problem
\bal\label{eq:ssc}
&{\bbeta^i} = \argmin_{\bbeta \in \RR^d} \frac{1}{2n}\ltwonorm{\bar \bX^i - \bar \bX \bbeta}^2 + \lambda \lonenorm{\bbeta}, \quad \bbeta_i = 0,
\eal%
for each $i \in [d]$. $\bar \bX^i$ is the $i$-th column of $\bar \bX$, which is also the $i$-th data point, $\lambda > 0$ is the weight for the $\ell^1$ regularization. Under certain conditions on $\bar \bX$ and the underlying subspaces, it is proved by \cite{Soltanolkotabi2012,WangX13} that nonzero elements of $\bbeta^*$ correspond to data points lying in the same subspace as $\bar\bX^i$, and in this case $\bbeta^*$ is said to satisfy the Subspace Detection Property (SDP). It has been proved that SDP is crucial for subspace recovery in the subspace clustering literature. By solving (\ref{eq:ssc}) for all $i \in [d]$, we have a sparse code matrix $\bbeta = [{\bbeta^1},{\bbeta^2},\ldots,{\bbeta^d}] \in \RR^{d \times d}$, and a sparse similarity matrix $\bW$ is constructed by $\bW = \frac{|\bX|+|\bX^{\top}|}{2}$. Spectral clustering is performed on $\bW$ to produce the final clustering result of Noisy SSC. Two measures are used to evaluate the performance of different clustering methods, i.e. the Accuracy (AC) and the Normalized Mutual Information (NMI) \cite{Zheng04}. In this experiment, we employ SRO to solve the sketched version of problem (\ref{eq:ssc}) with $\tilde n  = \frac{n}{15}$. Note that SRO is equivalent to Iterative SRO with $N=1$, and we do not incur more iterations by Iterative SRO since SRO produces satisfactory results. Figure~\ref{fig:SRO-subspace-clustering-acc} and Figure~\ref{fig:SRO-subspace-clustering-nmi} illustrate the accuracy (left) and NMI (right) of sketched Noisy SSC by SRO with respect to various choices of the regularization weight $\lambda$ on the Extended Yale-B Dataset. The Extended Yale-B Dataset contains face images for $38$ subjects with about $64$ frontal face images of size $32 \times 32$ taken under different illuminations for each subject, and $\bar \bX$ is of size $1024 \times 2414$. SRO-GSE stands for SRO with Gaussian subspace embedding, and SRO-SSE stands for SRO with sparse subspace embedding. We compare SRO-GSE and SRO-SSE to K-means (KM), Spectral Clustering (SC), Sparse Manifold Clustering and Embedding (SMCE) \cite{ElhamifarV11} and SSC by Orthogonal Matching Pursuit (SSC-OMP) \cite{Dyer13a}. It can be observed that SRO-GSE and SRO-SSE outperform other competing clustering methods by a notable margin, and they perform even better than the original Noisy SSC for most values of $\lambda$, due to the fact that sketching potentially reduces the adverse effect of noise in the original data.

\begin{minipage}{0.45\textwidth}
\begin{figure}[H]
\centering
\includegraphics[width=0.9\textwidth]{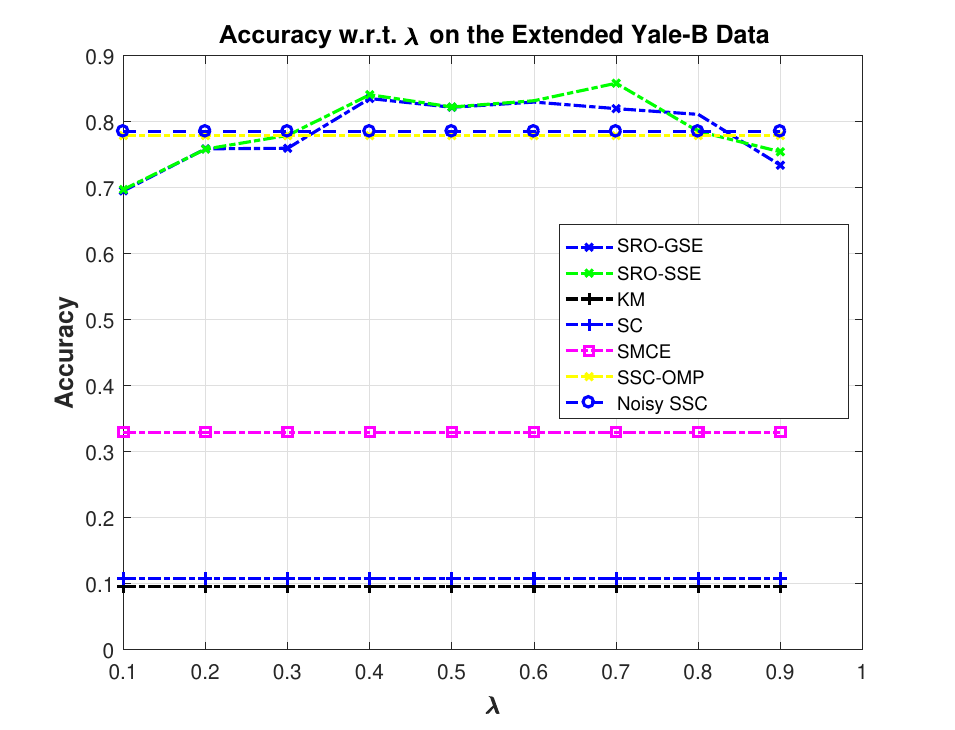}
\caption{ Accuracy with respect to different values of $\lambda$ on the Extended Yale-B Dataset}
\label{fig:SRO-subspace-clustering-acc}
\end{figure}
\end{minipage}
\begin{minipage}{0.45\textwidth}
\begin{figure}[H]
\centering
\includegraphics[width=0.9\textwidth]{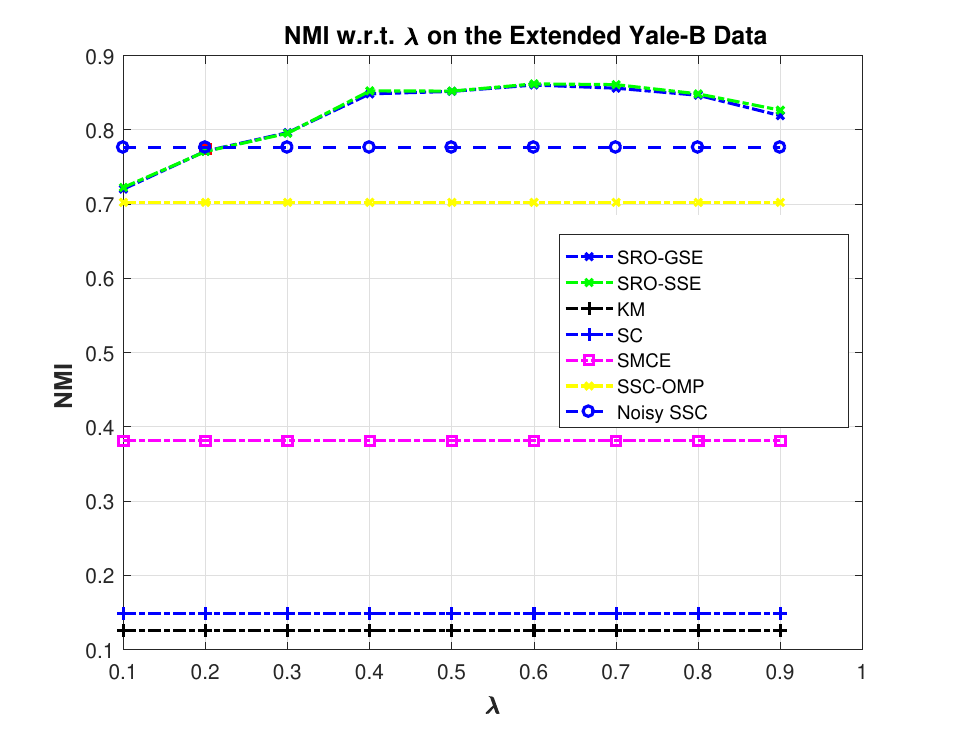}
\caption{ NMI with respect to different values of $\lambda$ on the Extended Yale-B Dataset}
\label{fig:SRO-subspace-clustering-nmi}
\end{figure}
\end{minipage}
\section{Proofs}
\label{sec::proofs}
We present proofs of theoretical results of the original paper in this section.

\subsection{Proofs for Section~\ref{sec::error-bounds} and
Section~\ref{sec::iterative-SRO}}
Before presenting the proof of Theorem~\ref{theorem::error-bound-general}, the following lemma is introduced which shows that subspace embedding approximately preserves inner product with high probability.

\begin{lemma}\label{lemma::innerprod-projection}
\normalfont
Suppose $\bP$ is a $(1 \pm \eps)$ $\ell^2$-subspace embedding for $\bX$, and let $\cC(\bX)$ denote the column space of $\bX$. Then with probability $1-\delta$, for any two vectors $\bu \in \cC(\bX)$, $\bv \in \cC(\bX)$,
\bal\label{eq:matrix-prod-projection}
&\abth{\bu^{\top} \bP^{\top} \bP \bv - \bu^{\top} \bv} \le \ltwonorm{\bu} \ltwonorm{\bv} \eps.
\eal%
\end{lemma}
\begin{proof}
If $\bu = \bzero$ or $\bv = \bzero$, then (\ref{eq:matrix-prod-projection}) holds trivially. Otherwise, let $\bu' = \frac{\bu}{\ltwonorm{\bu}}$, $\bv' = \frac{\bv}{\ltwonorm{\bv}}$, and $\bu', \bv' \in \cC(\bX)$. According to the definition of $(1 \pm \eps)$ $\ell^2$-subspace embedding for $\bX$, with probability $1-\delta$,
\bal
&(1-\eps)\ltwonorm{\bu'+\bv'}^2 \le \ltwonorm{\bP (\bu'+\bv')}^2 \le (1+\eps)\ltwonorm{\bu'+\bv'}^2, \label{eq:matrix-prod-projection-seg1} \\
&(1-\eps)\ltwonorm{\bu'-\bv'}^2 \le \ltwonorm{\bP (\bu'-\bv')}^2 \le (1+\eps)\ltwonorm{\bu'-\bv'}^2. \label{eq:matrix-prod-projection-seg2}
\eal%
Subtracting (\ref{eq:matrix-prod-projection-seg2}) from (\ref{eq:matrix-prod-projection-seg1}), we have
\bal\label{eq:matrix-prod-projection-seg3}
&\abth{\bu'^{\top} \bP^{\top} \bP \bv' - \bu'^{\top} \bv'} \le \eps,
\eal%
and (\ref{eq:matrix-prod-projection}) holds by scaling (\ref{eq:matrix-prod-projection-seg3}) by $\ltwonorm{\bu} \ltwonorm{\bv}$.
\end{proof}

\begin{proof}[\rm \bf{Proof of Theorem~\ref{theorem::error-bound-general}}]

By the optimality of $\tilde \bbeta^*$, we have
\bal\label{eq:tildez-optimality}
&\ltwonorm { \tbX^{\top} \tbX \tilde \bbeta^* - \by^{\top}\bX +  \bv } = 0
\eal%
for some $\bv \in {\tilde \partial h_{\lambda}}(\tilde \bbeta^*)$. In the sequel, we will also use ${\tilde \partial h} (\cdot)$ to indicate an element belonging to ${\tilde \partial h_{\lambda}} (\cdot)$ if no special note is made.

By the optimality of $\bbeta^*$, we have
\bal\label{eq:z-optimality}
&\ltwonorm{ \bX^{\top} \bX \bbeta^*  -\by^{\top}\bX +  \bu
}  = 0,
\eal%
where $\bu \in {\tilde \partial h_{\lambda}}(\bbeta^*)$.

Define $\Delta \defeq \tilde \bbeta^* - \bbeta^*$, $\tilde \Delta \defeq  \big({\tilde \partial h_{\lambda}}(\tilde \bbeta^*) - {\tilde \partial h_{\lambda}}(\bbeta^*) \big)$. By (\ref{eq:tildez-optimality}) and (\ref{eq:z-optimality}), we have
\bal\label{eq:optimal-rp-general-seg1}
& \ltwonorm{ \tbX^{\top} \tbX \tilde \bbeta^* +   \bv -
 \bX^{\top} \bX \bbeta^*  - \bu }  \nonumber \\
 &= \ltwonorm{  \tbX^{\top} \tbX \tilde \bbeta^* -  \bX^{\top} \bX \bbeta^* +  \tilde \Delta } = 0
\eal%

By Cauchy-Schwarz inequality,
\bal\label{eq:optimal-rp-general-seg2}
& \Delta^{\top} \pth{  \tbX^{\top} \tbX \tilde \bbeta^* -  \bX^{\top} \bX \bbeta^* +  \tilde \Delta } \le \ltwonorm{\Delta} \ltwonorm{  \tbX^{\top} \tbX \tilde \bbeta^* -
 \bX^{\top} \bX \bbeta^* +  \tilde \Delta } = 0.
\eal%

On the other hand, the LHS of (\ref{eq:optimal-rp-general-seg2}) can be written as
\bal\label{eq:optimal-rp-general-seg3}
& \Delta^{\top} \pth{  \tbX^{\top} \tbX \tilde \bbeta^* -
  \bX^{\top} \bX \bbeta^* +  \tilde \Delta }
  = \Delta^{\top} \pth{  \tbX^{\top} \tbX \tilde \bbeta^* -    \tbX^{\top} \tbX \bbeta^* }  +
\Delta^{\top} \pth{    \tbX^{\top} \tbX \bbeta^*-    \bX^{\top} \bX \bbeta^*  } +  \Delta^{\top} \tilde \Delta.
\eal%

By (\ref{eq:optimal-rp-general-seg2}) and (\ref{eq:optimal-rp-general-seg3}), we have
\bal\label{eq:optimal-rp-general-seg3'}
&\Delta^{\top} \pth{  \tbX^{\top} \tbX \tilde \bbeta^* -    \tbX^{\top} \tbX \bbeta^* }  +
\Delta^{\top} \pth{    \tbX^{\top} \tbX \bbeta^*-    \bX^{\top} \bX \bbeta^*  } \le - \Delta^{\top} \tilde \Delta
\eal%

Now we derive lower bounds for the two terms on the LHS of (\ref{eq:optimal-rp-general-seg3'}). First, we have

\bal\label{eq:optimal-rp-general-seg4}
&\Delta^{\top} \pth{  \tbX^{\top} \tbX \tilde \bbeta^* -    \tbX^{\top} \tbX \bbeta^* } = \Delta^{\top}  \tbX^{\top} \tbX \Delta  = \ltwonorm{\tbX \Delta}^2  \ge (1-\eps) \ltwonorm{\bX \Delta}^2.
\eal%

Moreover,
\bal\label{eq:optimal-rp-general-seg5}
& \Delta^{\top} \pth{  \tbX^{\top} \tbX \bbeta^* -  \bX^{\top} \bX \bbeta^* } = \Delta^{\top} \pth{  \tbX^{\top} \tbX-  \bX^{\top} \bX } \bbeta^* \ge - \eps \ltwonorm{\bX \Delta} \ltwonorm{\bX \bbeta^*}.
\eal%

Plugging (\ref{eq:optimal-rp-general-seg4}) and
(\ref{eq:optimal-rp-general-seg5}) in (\ref{eq:optimal-rp-general-seg3'}), we have
\bal\label{eq:optimal-rp-general-seg6}
& (1-\eps) \ltwonorm{\bX \Delta}^2 -  \eps \ltwonorm{ \bX \Delta } \ltwonorm{\bX \bbeta^*} \le - \Delta^{\top} \tilde \Delta.
\eal%

Moreover, by the definition of degree of nonconvexity in (\ref{eq:degree-nonconvexity}), \bal\label{eq:optimal-rp-general-seg11}
- \Delta^{\top} \tilde \Delta \le  \theta_{h}(\bbeta^*,\kappa) \ltwonorm{\tilde \bbeta^* - \bbeta^*} +  \kappa \ltwonorm{\tilde \bbeta^* - \bbeta^*}^2.
\eal
(\ref{eq:error-bound-general}) then follows by (\ref{eq:optimal-rp-general-seg6})
and (\ref{eq:optimal-rp-general-seg11}).

\end{proof}

\begin{proof} [\rm \bf{Proof of Corollary~\ref{corollary::error-bound-nonconvex}}]
Because the Frechet subdifferential of $h$ is $L_h$-smooth, that is, $$\sup_{\bu \in \tilde \partial h_{\lambda}(\bx), \bv \in \tilde \partial h_{\lambda}(\by)}\ltwonorm{\bu - \bv} \le L_h \ltwonorm{\bx-\by},$$
it can be verified that $\theta_{h_{\lambda}}(\bbeta^*, \kappa) \le 0$ with holds with $\kappa = L_h$. This is due to the fact that for any $\bs \in \RR^d$, $\bt \in \RR^d$, and let $\bu \in \tilde \partial h_{\lambda}(\bs)$,  $\bv \in \tilde \partial h_{\lambda}(\bt)$, then we have
\bal\label{eq:error-bound-nonconvex-seg1}
-(\bs-\bt)^{\top} (\bu - \bv) - L_h \ltwonorm{\bs-\bt}^2
&\le \ltwonorm{\bs - \bt} \ltwonorm{\bu - \bv} - L_h \ltwonorm{\bs-\bt}^2 \nonumber \\
&\le L_h \ltwonorm{\bs-\bt}^2 - L_h \ltwonorm{\bs-\bt}^2 = 0.
\eal%
Therefore,
\bsal\label{eq:error-bound-nonconvex-seg2}
&\theta_{h_{\lambda}}(\bt,L_h) =\sup_{\bs \in \RR^d, \bs \neq \bt} \{ -\frac{1}{\ltwonorm{\bs - \bt}}(\bs-\bt)^{\top} ({\tilde \partial h}(\bs) - {\tilde \partial h}(\bt)) - L_h \ltwonorm{\bs-\bt}\}\le 0.
\esal%

Moreover, since $\bX$ is nonsingular, $\ltwonorm{\tilde \bbeta^* - \bbeta^*}^2 \le \frac{\norm{\tilde \bbeta^* - \bbeta^*}{\bX}^2}{\sigma_{\min}^2(\bX)}$.

Plugging the above results in (\ref{eq:error-bound-general}) and setting $\kappa$ to $L_h$, we have
\bal\label{eq:error-bound-nonconvex-seg3}
(1-\eps) \norm{\tilde \bbeta^* - \bbeta^*}{\bX}^2 - \eps \norm{\tilde \bbeta^* - \bbeta^*}{\bX} \norm{\bbeta^*}{\bX}
&\le  \theta_{h}(\bbeta^*,L_h)\ltwonorm{\tilde \bbeta^* - \bbeta^*} +  L_h \ltwonorm{\tilde \bbeta^* - \bbeta^*}^2 \nonumber \\
&\le  \frac{  L_h}{\sigma_{\min}^2(\bX)} \cdot \norm{\tilde \bbeta^* - \bbeta^*}{\bX}^2,
\eal%
and (\ref{eq:error-bound-noncovnex}) immediately follows from (\ref{eq:error-bound-nonconvex-seg3}).
\end{proof}

\begin{proof}[\rm \bf {Proof of Theorem~\ref{theorem::error-bound-convex}}]
When $h$ is convex, then its Frechet differential coincides with its subdifferential. As a result, for any $\bu \in \tilde \partial h(\bs)$ and $\bv \in \tilde \partial h(\bt)$, we have
$(\bs - \bt)^{\top}(\bu - \bv) \ge 0$, and it follows that
\bal\label{eq:error-bound-convex-seg1}
&-(\bs - \bt)^{\top}(\bu - \bv) \le 0.
\eal%
With $\kappa=0$, (\ref{eq:error-bound-convex-seg1}) suggests that
\bal\label{eq:error-bound-convex-seg2}
&\theta_{h_{\lambda}}(\bt, 0) = \sup_{\bs \in \RR^d, \bs \neq \bt} \{ -\frac{1}{\ltwonorm{\bs - \bt}}(\bs-\bt)^{\top} ({\tilde \partial h_{\lambda}}(\bs) - {\tilde \partial h_{\lambda}}(\bt)) \} \le 0,
\eal%
By (\ref{eq:error-bound-general}) in Theorem~\ref{theorem::error-bound-general} and (\ref{eq:error-bound-convex-seg2}), setting $\kappa=0$, we have
\bal\label{eq:error-bound-convex-seg3}
&(1-\eps) \norm{\tilde \bbeta^* - \bbeta^*}{\bX}^2 - \eps \norm{\tilde \bbeta^* - \bbeta^*}{\bX} \norm{\bbeta^*}{\bX}
\le  \theta_{h_{\lambda}}(\bbeta^*,\kappa)\ltwonorm{\tilde \bbeta^* - \bbeta^*} +  \kappa \ltwonorm{\tilde \bbeta^* - \bbeta^*}^2 \le 0
\eal%
If $\norm{\tilde \bbeta^* - \bbeta^*}{\bX} \neq 0$, it follows from (\ref{eq:error-bound-convex-seg3}) that
\bal\label{eq:error-bound-convex-seg4}
&\norm{\tilde \bbeta^* - \bbeta^*}{\bX}  \le  \frac{\eps}{1-\eps} \norm{\bbeta^*}{\bX} .
\eal%

\end{proof}

\begin{proof} [\rm \bf {Proof of Theorem~\ref{theorem::error-bound-strongly-convex}}]
For strongly convex function $h$, we have
\bal\label{eq:error-bound-strongly-convex-seg1}
&-(\bs - \bt)^{\top}\big(\nabla h_{\lambda}(\bs) - \nabla h_{\lambda}(\bt)\big) \le -\sigma \ltwonorm{\bs - \bt}^2.
\eal%
Therefore, the degree of nonconvexity $\theta_{h_{\lambda}}(\bt, \kappa) \le 0$ with $\kappa=-\sigma$. Setting $\kappa=-\sigma$ in (\ref{eq:error-bound-general}), we have
\bal\label{eq:error-bound-strongly-convex-seg2}
(1-\eps) \norm{\tilde \bbeta^* - \bbeta^*}{\bX}^2 - \eps \norm{\tilde \bbeta^* - \bbeta^*}{\bX} \norm{\bbeta^*}{\bX}
&\le  \theta_{h_{\lambda}}(\bbeta^*,-\sigma)\ltwonorm{\tilde \bbeta^* - \bbeta^*} -  \sigma \ltwonorm{\tilde \bbeta^* - \bbeta^*}^2 \nonumber \\
&\le - \sigma \ltwonorm{\tilde \bbeta^* - \bbeta^*}^2.
\eal%

(\ref{eq:error-bound-strongly-convex-seg2}) indicates that the equation
$(1-\eps)x^2-\eps\norm{\bbeta^*}{\bX}x+ \sigma \ltwonorm{\tilde \bbeta^* - \bbeta^*}^2=0$ has a real root for $x$, so that we have
\bal\label{eq:error-bound-strongly-convex-seg3}
&4 \sigma \ltwonorm{\tilde \bbeta^* - \bbeta^*}^2 \cdot (1-\eps) \le \eps^2 \norm{\bbeta^*}{\bX}^2 \nonumber \\
&\Rightarrow \ltwonorm{\tilde \bbeta^* - \bbeta^*}^2 \le \frac{\eps^2 }{4 \sigma (1-\eps)} \cdot \norm{\bbeta^*}{\bX}^2.
\eal%

If $\norm{\tilde \bbeta^* - \bbeta^*}{\bX} \neq 0$, it also follows from (\ref{eq:error-bound-strongly-convex-seg2}) that
\bal\label{eq:error-bound-strongly-convex-seg4}
&\norm{\tilde \bbeta^* - \bbeta^*}{\bX}  \le  \frac{\eps}{1-\eps} \norm{\bbeta^*}{\bX} - \frac{\sigma \ltwonorm{\tilde \bbeta^* - \bbeta^*}^2}{(1-\eps) \norm{\tilde \bbeta^* - \bbeta^*}{\bX}}.
\eal%

\end{proof}

\begin{proof} [\rm \bf {Proof of Theorem~\ref{theorem::iterative-sketch}}]
This proof mostly follows from the proof of our main Theorem~\ref{theorem::error-bound-general}, Theorem~\ref{theorem::error-bound-convex} and Corollary~\ref{corollary::error-bound-nonconvex}.

Let $\tilde \bbeta^* = \bbeta^{(t)}$.
Define $\Delta \defeq \tilde \bbeta^* - \bbeta^*$, $\tilde \Delta \defeq  \big({\tilde \partial h_{\lambda}}(\tilde \bbeta^*) - {\tilde \partial h_{\lambda}}(\bbeta^*) \big)$. By repeating the proof of Theorem~\ref{theorem::error-bound-general} and Theorem~\ref{theorem::error-bound-convex}, with probability $1-\delta$,

\bal\label{eq:iterative-sketch-seg1}
(1-\eps) \ltwonorm{\bX \Delta}^2 -  \eps \ltwonorm{ \bX \Delta } \ltwonorm{\bX (\bbeta^* - \bbeta^{(t-1)})} \le - \Delta^{\top} \tilde \Delta.
\eal
It follows by the proof of Theorem~\ref{theorem::error-bound-convex} that
$\lambda \Delta^{\top} \tilde \Delta \le 0$. As a result, it follows by
(\ref{eq:iterative-sketch-seg1}) that
\bal\label{eq:iterative-sketch-seg2}
&\norm{\tilde \bbeta^* - \bbeta^*}{\bX}  \le \frac{ \eps}{1-\eps} \norm{\bbeta^*-\bbeta^{(t-1)}}{\bX}
\eal%
For a constant $0 < \rho < 1$, by choosing $\eps \le \frac{\rho}{\rho+1}$ in (\ref{eq:iterative-sketch-seg2}), we have
\bal\label{eq:iterative-sketch-seg3}
&\norm{ \bbeta^{(t)} - \bbeta^*}{\bX} \le \rho \norm{\bbeta^*-\bbeta^{(t-1)}}{\bX}
\eal%
for any $t \ge 1$. It follows that
\bal\label{eq:iterative-sketch-seg4}
&\norm{ \bbeta^{(N)} - \bbeta^*}{\bX} \le \rho^N \norm{\bbeta^*}{\bX}
\eal%
with $\bbeta^{(0)} = \bzero$. The same proof is applied for the case that $h$ is $L_h$-smooth and $\bX$ has full column rank with $\frac{ L_h}{\sigma_{\min}^2(\bX)} < 1-\eps$.
\end{proof}

\subsection{Proofs for Section~\ref{sec:sketching-sparse-signal-recovery}: Sketching for  Sparse Convex Learning}

In Theorem~\ref{theorem:sketching-convex-minimax}, we showed that for a low-rank data matrix $\bX$ with ran $r \ll n$, then the Iterative ROS described in Algorithm~\ref{alg:iterative-sketch-SRO} applied on the sketched problem (\ref{eq:optimization-general-rp}) achieves the parameter estimation error of the order $\cO\pth{\bar s \log d/n}$ under
Assumption~\ref{assumption:convex-sparse-recovery}.
As explained in \cite{YangWLEZ16-sparse-nonlinear-regression}, Assumption~\ref{assumption:convex-sparse-recovery}
is weaker than the Restricted Isometry Property (RIP) in compressed sensing \cite{CandesTao05}. We show that \textbf{there exists low-rank data matrix $\bX$ satisfying Assumption~\ref{assumption:convex-sparse-recovery}}.
We first show by Theorem~\ref{theorem:low-rank-RIP} that there exists a low-rank data matrix
$\bX$ which satisfies  $\textup{RIP}(\delta,s)$ for $\delta \in (0,1)$
and $s \in \NN$ with $\textup{RIP}(\delta,s)$ defined in Theorem~\ref{theorem:satisfying-assumption}(b).

We define the necessary notations in the following text. Let $\bA$ be a matrix
and $\cS$ be a set, we denote by $\bA_{\cS}$ the submatrix of $\bA$ with columns indexed by $\cS$. Similarly, for a vector $\bv$, $\bv_{\cS}$ is a vector formed by elements of $\bv$ indexed by $\cS$.

\begin{theorem}\label{theorem:low-rank-RIP}
Let $\delta \in (0,1)$
and $s \in \NN$. There exists a data matrix $\bX \in \RR^{n \times d}$ with rank not greater than $c_1 n$ for $c_1 \in (0,1)$
such that $\textup{RIP}(\delta,s)$ holds with probability at least
$1-2 d^s \pth{1+\frac{8}{\delta}}^{s}
\exp\pth{-nc_{1,2}(\delta)}-2 d^s\pth{1+\frac{8}{\delta}}^{s}\exp\pth{-n^{\alpha_0}}
$, where $\alpha_0 \in (0,1/2)$ is a positive constant, $ c_{1,2}(\delta)$ is a positive constant depending on $c_1$ and $\delta$. Here $n \ll d$, $\log d \ll n^{\alpha_0}$. $\textup{RIP}(\delta,s)$ means
that for all $\bv \in \RR^d, \lzeronorm{\bv} \le s$,
\bals
(1-\delta)\ltwonorm{\bv}^2 \le \bv \bX^{\top} \bX \bv
\le (1+\delta) \ltwonorm{\bv}^2.
\eals
\end{theorem}
\begin{proof}
We construct the data matrix by
\bal\label{eq:low-rank-RIP-X}
\bX = \bU \bU^{\top} \bOmega, \quad
\bU \in \RR^{n \times r}, \Omega \in \RR^{n \times d},
r = c_1 n,
\eal
where $\bU$ is an orthogonal matrix and $\bU^{\top} \bU = \bI_r$,
all the elements of $\bOmega$ are i.i.d. Gaussian random variables
with $\bOmega_{ij} \sim \cN(0,1/r)$ for $i \in [n], j \in [d]$.
It is clear that the rank of $\bX$ is bounded by $c_1n$.
We show that $\textup{RIP}(\delta,s)$

Define the set
\bals
\cF \defeq \bigcup_{\bS \subseteq [d], \abth{\bS} = s} \cF_{\bS}, \quad \cF_{\bS} \defeq  \set{\bu \in \RR^d \mid \ltwonorm{\bu}=1, \supp{\bu} \subseteq \bS} .
\eals

Given a set $\bS \subseteq [d], \abth{\bS} = s$, let $\bv \in \cF_{\bS}$ ,
we define functions
\bals
F(\bv) \defeq \frac 1{c_1 \ltwonorm{\bA  \bv}^2} \ltwonorm{\bU \bU^{\top} \bA  \bv}^2, \quad g(\bx) \defeq \frac{1}{c_1} \ltwonorm{\bU \bU^{\top} \bx}^2.
\eals
where $\bA = \sqrt{r} \bOmega$ and the elements of $\bA$ are i.i.d. standard
Gaussian random variables. It is clear that
$F(\bv) = g(\bA \bv/\ltwonorm{\bA \bv})$, and
$g$ is a $2/c_1$-Lipschitz function. Moreover,
\bals
\Expect{\bx \sim \textup{Unif}(\unitsphere{n-1})}{g(\bx)} &=
\frac{1}{c_1} \Expect{\bx \sim \textup{Unif}(\unitsphere{n-1})}{\bU \bU^{\top} \bx \bx^{\top} \bU \bU^{\top}} \nonumber \\
&= \frac{c_1 n}{c_1n} = 1.
\eals

Let $\bx =\bA \bv/\ltwonorm{\bA \bv}$, then
$\bx \sim \textup{Unif}(\unitsphere{n-1})$, and $F(\bv) = g(\bx)$.
It follows by
Lemma~\ref{lemma:concentration-lipschitz-unit-sphere} that
\bal\label{eq:low-rank-RIP-seg1}
\Prob{\abth{g(\bx) - \Expect{}{g(\bx)}} > t} \le 2\exp\pth{-n c^2_1 t^2/8},
\eal
It follows by (\ref{eq:low-rank-RIP-seg1})
that for a given $\bv \in \cF_{\bS}$,
$\Prob{\abth{F(\bv) - 1} \le t} \ge 1-2\exp\pth{-n c^2_1 t^2/8}$.

On the other hand, $\bv^{\top} \bX^{\top} \bX \bv =
1/r \cdot\ltwonorm{\bU \bU^{\top} \bA  \bv}^2 $. Noting that
$\ltwonorm{\bA  \bv}^2$ is $\chi^2$, it follows by standard concentration on $\chi^2$ that for a given $\bv \in \cF_{\bS}$, with probability at least
$1-2 \exp\pth{-nc_{1,2}(\delta)}-2 \exp\pth{-n^{\alpha_0}}$, we have
$\abth{\bv^{\top} \bX^{\top} \bX \bv-1} \le \delta/2$, and
$\abth{\ltwonorm{\bX \bv}-1} \le \delta/2$.

It follows by \cite[Lemma 5.2]{Vershynin2012-nonasymptotics-random-matrix}, there exists an $\delta/4$-net $N_{\delta/4}(\cF_{\bS}, \ltwonorm{\cdot}) \subseteq \cF_{\bS}$ of $\cF_{\bS}$ such that $\abth{N_{\delta/4}(\cF_{\bS}, \ltwonorm{\cdot})} \le \pth{1+\frac{8}{\delta}}^{s}$. Using the union bound, with probability at least
$1-2\pth{1+\frac{8}{\delta}}^{s}
\exp\pth{-n c_{1,2}(\delta)}-2 \pth{1+\frac{8}{\delta}}^{s}\exp\pth{-n^{\alpha_0}}$,
the event that $\abth{\ltwonorm{\bX \bv}-1} \le \delta/2$ holds for all
$\bv \in N_{\delta/4}(\cF_{\bS}, \ltwonorm{\cdot})$ happens. Define this event as $\cA_{\bS}$, and the following argument is conditioned on $\cA_{\bS}$.

We define $A>0$ as the smallest number such that
$\ltwonorm{\bX \bv} \le 1+A$ for all $\bv \in \cF_{\bS}$.
For any $\bv \in \cF_{\bS}$, there exists $\bv' \in N_{\delta/4}(\cF_{\bS}, \ltwonorm{\cdot})$ such that $\ltwonorm{\bv-\bv'} \le \delta/4$. As a result,
\bals
\ltwonorm{\bX \bv} \le \ltwonorm{\bX \bv'} + \ltwonorm{\bX (\bv-\bv')}
\le 1+ \frac \delta 2 + (1+A) \frac \delta 4.
\eals
It follows by the definition of $A$ that $A \le \delta/2 + (1+A) \delta/4$,
so $A \le \frac{3\delta/4}{1-\delta/4} \le \delta$. We then have
$\ltwonorm{\bX \bv} \le 1+\delta$ for all $\bv \in \cF_{\bS}$. On the other
hand,
\bals
\ltwonorm{\bX \bv} \ge \ltwonorm{\bX \bv'} - \ltwonorm{\bX (\bv-\bv')}
\ge1-\frac \delta 2 -(1+\delta) \frac \delta 4  \ge 1-\delta,
\quad \forall \bv \in \cF_{\bS}.
\eals

As a result, conditioned on the event $\cA_{\bS}$ for a set $\bS \subseteq [d], \abth{\bS} = s, \abth{\ltwonorm{\bX \bv}-1} \le \delta$ holds for all $\bv \in \cF_{\bS}$. Because $\abth{\cF } \le d^s$, using the union bound, with probability specified in this theorem,
$\abth{\ltwonorm{\bX \bv}-1} \le \delta$ holds for all  $\bv \in \RR^d, \lzeronorm{\bv} \le s$.

\end{proof}

\textbf{Low-Rank Matrix $\bX$ Satisfying Assumption~\ref{assumption:convex-sparse-recovery}}
It is proved in Theorem~\ref{theorem:low-rank-RIP} that the low-rank data matrix $\bX$ constructed by (\ref{eq:low-rank-RIP-X}) satisfies $\textup{RIP}(\delta,s)$
for $\delta \in (0,1)$ and $s \in \NN$. As indicated by \cite{YangWLEZ16-sparse-nonlinear-regression},
The condition (\ref{eq:convex-sparse-recovery-sparse-eigenvalue-cond}) in Assumption~\ref{assumption:convex-sparse-recovery} is weaker than RIP. To see this, (\ref{eq:convex-sparse-recovery-sparse-eigenvalue-cond})
holds with $k^* = (s-\bar s)/2$ if $\textup{RIP}(\delta,s)$ holds with $s = 5 \bar s$ and $\delta = 1/3$.

We need the following lemma for the proof of
Theorem~\ref{theorem:sketching-convex-minimax}.
\begin{lemma}
[{\cite[Lemma 5]{YangWLEZ16-sparse-nonlinear-regression}}]
\label{lemma:quadratic-lower-bound}
For any $\bv \in \RR^d$ and any index set $\cS \subseteq [d]$
with $\abth{\cS} = {\bar s}$. Let $\cJ$ be the set of indices of the largest
$k^*$ elements of $\bv_{\cS^c}$ in absolute value and let $\cI = \cJ \bigcup
\cS$. Here ${\bar s}$ and $k^*$ are the same as those in
Assumption~\ref{assumption:convex-sparse-recovery}. Assume that
$\lonenorm{\bv_{\cS^c}} \le \gamma \lonenorm{\bv_{\cS}}$ for some $\gamma > 0$. Then we have $\ltwonorm{\bv} \le (1+\gamma) \ltwonorm{\bv_{\cI}}$, and
\bal\label{eq:quadratic-lower-bound}
\bv \bX^{\top} \bX \bv \ge \rho_{\cL,-}({\bar s}+k^*) \cdot
\pth{\ltwonorm{\bv_{\cI}} - \gamma \sqrt{{\bar s}/k^*}
\sqrt{\rho_{\cL,+}(k^*)/\rho_{\cL,-}({\bar s}+2k^*)-1} \ltwonorm{\bv_{\cS}}}
\ltwonorm{\bv_{\cI}}.
\eal
\end{lemma}

\begin{proof}[\textup{\bf Proof of
Theorem~\ref{theorem:sketching-convex-minimax}}]

We consider the upper bound for the quadratic form
$\iprod{\nabla \cL(\tilde \bbeta^*) - \nabla \cL(\bar \bbeta)}{\tilde \bbeta^*
-\bar \bbeta}=\Delta^{\top} \bX^{\top} \bX \Delta$, where $\tilde \bbeta^* = \beta^{(N)}$ and $\Delta \defeq \tilde \bbeta^* -\bar \bbeta$. We define
$\cS = \supp{\bar \bbeta}$. Let $\bv$ be a vector, we denote by $\bv_{\cS}$ the vector formed by elements of $\bv$ in the set $\cS$.

We have
\bal\label{eq:sketching-convex-minimax-seg1}
\Delta^{\top} \bX^{\top} \bX \Delta &= \iprod{\nabla \cL(\tilde \bbeta^*) - \nabla \cL(\bar \bbeta)}{\tilde \bbeta^*-\bar \bbeta} \nonumber \\
&\le
\iprod{\nabla \cL(\tilde \bbeta^*)}{\tilde \bbeta^*-\bar \bbeta}
+\supnorm{\nabla \cL(\bar \bbeta)} \lonenorm{\Delta} \nonumber \\
&\stackrel{\circled{1}}{\le} \iprod{\nabla \cL(\tilde \bbeta^*) - \nabla \cL(\bbeta^*)}{\tilde \bbeta^*-\bar \bbeta} + \iprod{\nabla \cL(\bbeta^*)}{\tilde \bbeta^*-\bar \bbeta}
+ 2 \sigma \sqrt{\frac{\log d}{n}}
 \lonenorm{\Delta} \nonumber \\
&\le \supnorm{\nabla \cL(\tilde \bbeta^*) - \nabla \cL(\bbeta^*)}\lonenorm{\Delta}
-\lambda \lonenorm{\Delta_{\cS^c}}
+ \lambda \lonenorm{\Delta_{\cS}} +\frac{2}{c} \lambda\lonenorm{\Delta}
\nonumber \\
&\stackrel{\circled{2}}{\le} \lambda \mu \lonenorm{\Delta}
-\lambda \lonenorm{\Delta_{\cS^c}}
+ \lambda \lonenorm{\Delta_{\cS}} + \frac{2}{c} \lambda\lonenorm{\Delta} \nonumber \\
& =-\lambda \pth{1-\mu - \frac 2c} \lonenorm{\Delta_{\cS^c}}
+ \lambda \pth{1+ \mu + \frac 2c}\lonenorm{\Delta_{\cS}}.
\eal
Here $\circled{1}$ follows by Lemma~\ref{lemma:bounded-grad-L-barbeta}, and $\rho^N \ltwonorm{\bX} \norm{\bbeta^*}{\bX} \le \lambda \mu$
in $\circled{2}$.

It follows by (\ref{eq:sketching-convex-minimax-seg2}) that
$\lonenorm{\Delta_{\cS^c}} \le \pth{(1+ \mu + 2/c)/(1-\mu - 2/c)}\lonenorm{\Delta_{\cS}}$.

We now apply Lemma~\ref{lemma:quadratic-lower-bound} with $\bv = \Delta$
and $\gamma =\pth{(1+ \mu + 2/c)/(1-\mu - 2/c)}$. It follows by Lemma~\ref{lemma:quadratic-lower-bound}
that $\ltwonorm{\Delta} \le (1+\gamma) \ltwonorm{\Delta_{\cI}}$. Moreover,
it follows by Assumption~\ref{assumption:convex-sparse-recovery}
that $\rho_{\cL,+}(k^*) / \rho_{\cL,-}(2k^*+{\bar s}) \le 1+0.5 k^*/{\bar s}$. Plugging this inequality in (\ref{eq:quadratic-lower-bound}), we have
\bal\label{eq:sketching-convex-minimax-seg2}
\Delta \bX^{\top} \bX \Delta &\ge \rho_{\cL,-}({\bar s}+k^*) \cdot
\pth{\ltwonorm{\Delta_{\cI}} - \gamma \sqrt{{\bar s}/k^*}
\sqrt{\rho_{\cL,+}(k^*)/\rho_{\cL,-}({\bar s}+2k^*)-1} \ltwonorm{\Delta_{\cS}}}
\ltwonorm{\Delta_{\cI}} \nonumber \\
&\ge\rho_{\cL,-}({\bar s}+k^*) \cdot \pth{1- \gamma \sqrt{0.5}}  \ltwonorm{\Delta_{\cI}}^2.
\eal
It follows from (\ref{eq:sketching-convex-minimax-seg1})
that $\Delta^{\top} \bX^{\top} \bX \Delta \le \lambda \pth{1+ \mu + \frac 2c}\lonenorm{\Delta_{\cS}} \le \lambda \pth{1+ \mu + \frac 2c}
\sqrt{{\bar s}} \ltwonorm{\Delta_{\cI}}$. It follows by this inequality and
(\ref{eq:sketching-convex-minimax-seg2}) that
\bals
\ltwonorm{\Delta} \le \ltwonorm{\Delta_{\cI}} \le
\frac{\pth{1+ \mu + \frac 2c}
\sqrt{{\bar s}} \lambda }{\rho_{\cL,-}({\bar s}+k^*) \cdot \pth{1- \gamma \sqrt{0.5}}  },
\eals
which proves (\ref{eq:sketching-convex-minimax}).

\end{proof}

\subsection{Sketching for  Sparse Nonconvex Learning}

\begin{proof}[\textup{\bf Proof of Corollary~\ref{corollary::error-bound-noncovnex-sparse-eigenvalue}}]

We have $h_{\lambda}(\bbeta) = \lambda \lonenorm{\bbeta}+Q_{\lambda}(\bbeta)$. For any $\bs \in \RR^d$, $\bt \in \RR^d$, and let $\bu \in \tilde \partial h_{\lambda}(\bs)$, $\bv \in \tilde \partial h_{\lambda}(\bt)$, then
$\bu = \xi_1 + \nabla Q_{\lambda}(\bs)$ and $\bv = \xi_2 + \nabla Q_{\lambda}(\bt)$
with $\bx_1 \in \partial \pth{\lambda \lonenorm{\bs}}$ and $\bx_2 \in \partial \pth{\lambda \lonenorm{\bt}}$. We have
\bal\label{eq:error-bound-noncovnex-sparse-eigenvalue-seg1}
-(\bs-\bt)^{\top} (\bu - \bv)
&= - (\bs - \bt)^{\top} (\xi_1 -\xi_2) -(\bs-\bt)^{\top} \pth{\nabla Q_{\lambda}(\bs) - \nabla Q_{\lambda}(\bt)} \nonumber \\
&\stackrel{\circled{1}}{\le}-(\bs-\bt)^{\top} \pth{\nabla Q_{\lambda}(\bs) - \nabla Q_{\lambda}(\bt)} \nonumber \\
&\stackrel{\circled{2}}{\le} \zeta_{-} \ltwonorm{\bs-\bt}^2,
\eal
where $\circled{1}$ follows from the convexity of $\lambda \lonenorm{\cdot}$, and $\circled{2}$ follows from the regularity condition (b) in
Assumption~\ref{assumption:q-regularity-cond}.

It follows by (\ref{eq:error-bound-noncovnex-sparse-eigenvalue-seg1}) that the degree of nonconvexity
$\theta_{h_{\lambda}}(\bt,\zeta_{-} ) \le 0$. Using the upper bound for $\tilde n$ in the given condition, plugging $\theta_{h_{\lambda}}(\bt,\zeta_{-} ) \le 0$ in
 (\ref{eq:error-bound-general}) and setting $\kappa$ to $\zeta_{-}$, we have
\bal\label{eq:error-bound-noncovnex-sparse-eigenvalue-seg2}
(1-\eps) \norm{\tilde \bbeta^* - \bbeta^*}{\bX}^2 - \eps \norm{\tilde \bbeta^* - \bbeta^*}{\bX} \norm{\bbeta^*}{\bX}
&\le  \theta_{h}(\bbeta^*,L_h)\ltwonorm{\tilde \bbeta^* - \bbeta^*} +\zeta_{-}  \ltwonorm{\tilde \bbeta^* - \bbeta^*}^2 \nonumber \\
&\le \zeta_{-}  \ltwonorm{\tilde \bbeta^* - \bbeta^*}^2.
\eal
By the definition of sparse eigenvalues in Definition~\ref{def:sparse-eigenvalue}, we have
\bals
\norm{\tilde \bbeta^* - \bbeta^*}{\bX}^2 \ge \rho_{\cL,-}(s)\ltwonorm{\tilde \bbeta^* - \bbeta^*}^2, \nonumber \\
\norm{\tilde \bbeta^* - \bbeta^*}{\bX}^2 \le \rho_{\cL,+}(s)\ltwonorm{\tilde \bbeta^* - \bbeta^*}^2.
\eals

Applying the above inequalities in (\ref{eq:error-bound-noncovnex-sparse-eigenvalue-seg2}), we have
\bals
(1-\eps)\rho_{\cL,-}(s) \ltwonorm{\tilde \bbeta^* - \bbeta^*}^2
\le \eps \sqrt{\rho_{\cL,+}(s)} \ltwonorm{\tilde \bbeta^* - \bbeta^*} \norm{\bbeta^*}{\bX} + \zeta_{-}  \ltwonorm{\tilde \bbeta^* - \bbeta^*}^2
\eals
which leads to (\ref{eq:error-bound-noncovnex-sparse-eigenvalue}).
\end{proof}

\begin{definition}\label{def:eps-net}
($\eps$-net) Let $(X, d)$ be a metric space and let $ \eps > 0$. A subset $N_{\eps}(X,d)$ is called an $\eps$-net of $X$ if for every point $x \in X$, there exists some point $y \in N_{\eps}(X,d)$ such that $d(x,y) \le \eps$. The minimal cardinality of an $\eps$-net of $X$, if finite, is called the covering number of $X$ at scale $\eps$.
\end{definition}

\begin{theorem}
[\hspace{-.1pt}{revised version of \cite[Theorem 2.1]{Woodruff2014-skeching-numerical-algebra}} with explicit constant in the bound for $\tilde n$]
\label{theroem:JLT}
Let $0 < \eps,\delta < 1$ and $\bP \in \RR^{\tilde n \times n}$ be a Gaussian subspace embedding defined in Definition~\ref{def:gaussian-subspace-embedding}. Then if
$\tilde n \ge 32\eps^{-2}\log(2f_0/\delta)$, $\bP$ is a $\textup{JLT}(\eps,\delta,f_0)$. That is, for any $f_0$-element set $\cV$ and all $\bv,\bv' \in \cV$, with probability at least $1-\delta$, it holds that $\abth{\iprod{\bP\bv}{\bP\bv'} - \iprod{\bv}{\bv'}} \le \eps
\ltwonorm{\bv}\ltwonorm{\bv'}$.
\end{theorem}

\begin{lemma}
\label{lemma:sketching-satisfying-assumption}
Let $s \in [d]$, $s \ge 2$, $0 < \eps,\delta < 1$, and $\bP \in \RR^{\tilde n \times n}$ be a Gaussian subspace embedding defined in Definition~\ref{def:gaussian-subspace-embedding}. If
$\tilde n \ge c_0 \eps^{-2}\pth{\log(2/\delta) + s\log d + s \log 5
+1/d^{s-1} } $ where $c_0$ is a positive constant, then with probability at least $1-\delta$, the following inequalities hold:
\bal
&\rho_{\tilde \cL,+}(s) \le \rho_{\cL,+}(s) + \eps \sqrt{s}, \rho_{\tilde \cL,-}(s) \ge \rho_{\cL,-}(s) - \eps \sqrt{s}, \label{eq:sketched-sparse-eigenvalue}
\\
&\supnorm{\nabla \tilde \cL(\bar \bbeta)}
\le \supnorm{\nabla \cL(\bar \bbeta)}  + \eps \sqrt{{\bar s}} \ltwonorm{\bar \bbeta}.
\label{eq:sketched-supnorm-gradient-L-barbeta}
\eal
\end{lemma}

\begin{proof}[\textup{\bf Proof of Lemma~\ref{lemma:sketching-satisfying-assumption} }]
Define the set
\bals
\cF \defeq \bigcup_{\bS \subseteq [d], \abth{\bS} = s} \cF_{\bS},
\quad
\cF_{\bS} \defeq  \set{\bu \in \RR^d \mid \supp{\bu} \subseteq \bS}.
\eals
Then $\abth{\cF} = {d \choose s} < d^{s}$. We now consider the subspace
\bals
\cU_{\bS} \defeq \set{\bu = \bX \bv \mid \ltwonorm{\bu} = 1, \bu = \bX \bv \textup{ for some } \bv \in \cF_{\bS}}, \quad \bS \subseteq [d], \abth{\cS} = s.
\eals
It can be verified that the dimension of $\cU_{\bS}$ is not greater than $s$. Let $\gamma > 0$, it follows by \cite[Lemma 5.2]{Vershynin2012-nonasymptotics-random-matrix}, there exists an $\gamma$-net $N_{\gamma}(\cU_{\bS}, \ltwonorm{\cdot}) \subseteq \cU_{\bS}$ of $\cU_{\bS}$ such that $\abth{N_{\gamma}(\cU_{\bS}, \ltwonorm{\cdot})} \le \pth{1+\frac{2}{\gamma}}^{s}$.

Define the set
\bals\cV_s \defeq \bigcup_{\bS \subseteq [d], \abth{\bS} = s}  N_{1/2}(\cU_{\bS}, \ltwonorm{\cdot})  \bigcup  \pth{\bigcup_{i \in [d]} N_{1/2}(\be^{(i)}, \ltwonorm{\cdot})},
\eals
where $\be^{(i)}$ is the subspace spanned by $\set{\be^i, \bar \bbeta}$
where $\set{\be^i}_{i=1}^d$ is the standard basis of $\RR^d$. As a result of the above argument, $f_0 \defeq \abth{\cV_s} \le{d \choose s} \cdot 5^s + 25d$. It follows by standard calculation that  $\log(2/\delta) + s\log d + s \log 5 +1/d^{s-1} \ge \log(2f_0/\delta)$. It then follows by
Theorem~\ref{theroem:JLT}  that when
$\tilde n \ge c_0\eps^{-2}\pth{\log(2/\delta) + s\log d + s \log 5 +1/d^{s-1}}
\ge c_0\eps^{-2}\log(2f_0/\delta)$ with $c_0 = 512$, $\bP$ is a $\textup{JLT}(\eps/4,\delta,f_0)$. Therefore, with probability at least $1-\delta$, $\bP$ is a $(1\pm \eps/4)$ $\ell^2$-embedding for $\cV_s$.

For any $\bu \in \cU_{\bS}$, there exists a $\bu' \in N_{\gamma}(\cU_{\bS}, \ltwonorm{\cdot})$ such that $\ltwonorm{\bu-\bu'} \le \gamma$. We now use the simpler notation $\iprod{\bP \bu'}{\bP \bv'} = (1 \pm \eps) \iprod{\bu'}{\bv'}$ to indicate $ \abth{
\iprod{\bP \bu'}{\bP \bv'} - \iprod{\bu'}{\bv'}} \le \eps \ltwonorm{\bu'}
\ltwonorm{\bv'}$ for any two numbers $\bu',\bv'$ in the following text.

We follow the construction in the proof of
\cite[Theorem 2.3]{Woodruff2014-skeching-numerical-algebra}, that is, we build a $1/2$-net $N_{1/2}(\cU_{\bS}, \ltwonorm{\cdot}) \subseteq \cU_{\bS}$ of $\cU_{\bS}$. Given any $\bS \subseteq [d], \abth{\cS} = s$, we now show in the sequel that if $\bP$ is a $(1\pm \eps/4)$ $\ell^2$-embedding for $N_{1/2}(\cU_{\bS}, \ltwonorm{\cdot})$, that is, $ \iprod{\bP \bu'}{\bP \bv'} = (1 \pm \eps/4) \iprod{\bu'}{\bv'}$
for all $\bu',\bv' \in N_{1/2}(\cU_{\bS}, \ltwonorm{\cdot})$, then $\bP$ is also a $(1\pm \eps)$ $\ell^2$-embedding for $\cU_{\bS}$, that is,
$\iprod{\bP \bu}{\bP \bv} = (1 \pm \eps) \iprod{\bu}{\bv}$ for all
$\bu,\bv \in \cU_{\bS}$. To see this, any $\bu \in \cU_{\bS}$ can be expressed by
 \bal\label{eq:sketching-satisfying-assumption-seg1}
\bu = \sum\limits_{j=0}^{\infty} \bu^j,
\quad \ltwonorm{\bu^j} \le \frac 1{2^j}, \bu^j \in N_{1/2}(\cU_{\bS}, \ltwonorm{\cdot}).
\eal
(\ref{eq:sketching-satisfying-assumption-seg1}) can be proved by induction. By the definition of $1/2$-net, there exists $\bu^0 \in N_{1/2}(\cU_{\bS}, \ltwonorm{\cdot})$ such that
$\ltwonorm{\bu-\bu^0} \le 1/2$ and $\bu = \bu^0 + (\bu-\bu^0)$. Also, there exists $\bu'^1 \in  N_{1/2}(\cU_{\bS}, \ltwonorm{\cdot})$ such that
$\ltwonorm{((\bu-\bu^0))/\ltwonorm{(\bu-\bu^0)}-\bu'^1} \le 1/2$, so that $\bu = \bu^0 + \bu^1 + \pth{\bu - \bu^0 - \bu^1 }$ with $\bu^1 = \ltwonorm{(\bu-\bu^0)}\bu'^1$,
$\ltwonorm{\bu^1} \le 1/2$, and $\ltwonorm{\bu - \bu^0 - \bu^1} \le 1/2^2$.
As the induction step, if $\bu = \sum\limits_{j=0}^{k} \bu^j + \pth{\bu - \sum\limits_{j=1}^{k} \bu^j}$ for $k \ge 1$ such that $\ltwonorm{\bu^j} \le \frac 1{2^j}$ for all $j \in [0,k]$ and $\ltwonorm{\bu - \sum\limits_{j=1}^{k} \bu^j} \le 1/2^{k+1}$. Then there exists $\bu'^{k+1} \in  N_{1/2}(\cU_{\bS}, \ltwonorm{\cdot})$ such that  $\ltwonorm{\pth{\bu - \sum\limits_{j=1}^{k} \bu^j}/\ltwonorm{\bu - \sum\limits_{j=1}^{k} \bu^j}-\bu'^{k+1}} \le 1/2$. As a result,
$\bu = \sum\limits_{j=0}^{k+1} \bu^{k+1} + \pth{\bu - \sum\limits_{j=1}^{k+1} \bu^j}$ where $\bu^{k+1} = \ltwonorm{\bu - \sum\limits_{j=1}^{k} \bu^j}\bu'^{k+1}$,
 $\ltwonorm{\bu^{k+1}} \le \ltwonorm{\bu - \sum\limits_{j=1}^{k} \bu^j} \le 1/2^{k+1}$ and $\ltwonorm{\bu - \sum\limits_{j=1}^{k+1} \bu^j} \le 1/2^{k+2}$.

 It follows  by (\ref{eq:sketching-satisfying-assumption-seg1}) that
\bal\label{eq:sketching-satisfying-assumption-seg2}
\ltwonorm{\bP \bu}^2 &= \ltwonorm{\bP \pth{\sum\limits_{j=0}^{\infty} \bu^j}}^2
\nonumber \\
&=\sum\limits_{j=0}^{\infty} \ltwonorm{\bP \bu^j}^2 +
2 \sum\limits_{i < j} \iprod{\bP \bu^i} {\bP \bu^j} \nonumber \\
&= \sum\limits_{j=0}^{\infty} \pth{1 \pm c }\ltwonorm{\bu^j}^2
+ 2 \sum\limits_{i < j} \iprod{\bu^i}{\bu^j} \pm
\frac {\eps}4 \sum\limits_{i < j} \ltwonorm{\bu^i}\ltwonorm{\bu^j} \nonumber \\
&= \pth{ \sum\limits_{j=0}^{\infty} \ltwonorm{\bu^j}^2
+ 2 \sum\limits_{i < j} \iprod{\bu^i}{\bu^j} } \pm
\frac {\eps}4 \pth{\sum\limits_{j=0}^{\infty} \ltwonorm{\bu^j}^2
+ 2 \sum\limits_{i < j} \ltwonorm{\bu^i}\ltwonorm{\bu^j} } \nonumber \\
&= \ltwonorm{\bu}^2 \pm \eps \ltwonorm{\bu}^2 = 1 \pm \eps.
\eal
It follows by (\ref{eq:sketching-satisfying-assumption-seg2}) that $\bP$ is a $(1\pm \eps)$ $\ell^2$-embedding for $\cU_{\bS}$.
Recall that $\bP$ is a $(1\pm \eps/4)$ $\ell^2$-embedding for $\cV_s \supseteq \bigcup_{\bS \subseteq [d], \abth{\cS} = s}  N_{1/2}(\cU_{\bS}, \ltwonorm{\cdot})$. By the above argument, $\bP$ is also a $(1\pm \eps)$ $\ell^2$-embedding
for $\cU_{\bS}$ for any $\bS \subseteq [d], \abth{\bS} = s$.

For any $\bv \in \RR^d$ such that $\lzeronorm{\bv} \le s, \ltwonorm{\bv}=1$, there must exists a set $\bS \subseteq [d], \abth{\bS} = s$ such that
$\bX \bv/\ltwonorm{\bX \bv} \in \cU_{\bS}$. As a result,
 \bal\label{eq:sketching-satisfying-assumption-seg3}
\abth{\bv^{\top} \bX^{\top} \bX \bv - \bv^{\top} \tbX^{\top} \tbX \bv}
\le \eps \ltwonorm{\bX \bv}^2 \le \eps \sqrt{s},
\eal
where the last inequality follows by $\max_{i \in [d]} \ltwonorm{\bX^i} \le 1$. Therefore, $\rho_{\tilde \cL,+}(s) \le \rho_{\cL,+}(s) + \eps \sqrt{s}$, and $\rho_{\tilde \cL,-}(s) \ge \rho_{\cL,-}(s) - \eps \sqrt{s}$.

Because $\bP$ is a $(1\pm \eps/4)$ $\ell^2$-embedding for $\cV_s \supseteq \bigcup_{i \in [d]} N_{1/2}(\be^{(i)}, \ltwonorm{\cdot})$, $\bP$ is also $(1\pm \eps)$ $\ell^2$-embedding for $\be^{(i)}$ for all $i \in [d]$. Therefore, we have
\bals
\sup_{i \in [d]} {\be^i}^{\top} \pth{\bX^{\top} \bX \bar \bbeta - {\tbX}^{\top} \tbX} \bar \bbeta \le \eps \ltwonorm{\bX \be_i} \ltwonorm{\bX \bar \bbeta},
\eals
so that
\bals
\supnorm{\nabla \tilde \cL(\bar \bbeta) - \nabla \tilde \cL(\bar \bbeta)}
&=\supnorm{\bX^{\top} \bX \bar \bbeta - {\tbX}^{\top} \tbX \bar \bbeta}
=\sup_{i \in [d]} {\be^i}^{\top} \pth{\bX^{\top} \bX \bar \bbeta - {\tbX}^{\top} \tbX} \bar \bbeta \nonumber \\
&\le \eps \ltwonorm{\bX\be^i} \ltwonorm{\bX \bar \bbeta}
\le \eps \sqrt{{\bar s}} \ltwonorm{\bar \bbeta}.
\eals

\end{proof}

\begin{theorem}
\end{theorem}
\begin{proof}
[\textup{\bf Proof of Theorem~\ref{theorem:satisfying-assumption} }]
It follows by (\ref{eq:sketched-sparse-eigenvalue}) in Lemma~\ref{lemma:sketching-satisfying-assumption} that
\bals
 \rho_{\tilde \cL,-}(s) \ge \rho_{\cL,-}(s) - \eps \sqrt{s} > 0
\eals
with $s = s_0 = {\bar s}+2 \tilde s$.
Moreover, it also follows by (\ref{eq:sketched-sparse-eigenvalue}) that
\bals
C' \rho_{\tilde \cL,-}({\bar s}+2\tilde s) \ge
C' \pth{\rho_{\cL,-}({\bar s}+2\tilde s)- \eps \sqrt{{\bar s}+2\tilde s}}
\ge \zeta_{-}.
\eals
In addition, by (\ref{eq:sketched-sparse-eigenvalue}) we have
\bals
\kappa' = \frac{\rho_{\cL,+}(s_0) + \eps \sqrt{s_0}- \zeta_+}{\rho_{\cL,-}(s_0) -\eps \sqrt{s_0} - \zeta_{-}}
\ge \frac{\rho_{\tilde \cL,+}(s_0)- \zeta_+}{\rho_{\tilde \cL,+}(s_0)- \zeta_{-}}
= \tilde \kappa
\eals
with $\tilde \kappa$ specified in
Assumption~\ref{assumption:nonconvex-sketch-sparse-recovery}.
Therefore, if $\pth{144 \kappa'^2 + 250 \kappa'} {\bar s} < \tilde s$, we have $\pth{144 \tilde \kappa^2 + 250 \tilde \kappa} \bar s = \tilde C \bar s < \tilde s$
which completes the first part of the claim.

Now we need to verify that either condition (a) or condition (b) can lead to the following conditions:
\bal
\rho_{\cL,-}(s_0) &> \eps \sqrt{s_0}, \label{eq:satisfying-assumption-conds-1} \\
\zeta_{-} &\le C' \pth{\rho_{\cL,-}(s_0)- \eps \sqrt{s_0}},  \label{eq:satisfying-assumption-conds-2} \\
\pth{144 \kappa'^2 + 250 \kappa'} {\bar s} &< \tilde s, \quad
\kappa' = (\rho_{\cL,+}(s_0) + \eps \sqrt{s_0}- \zeta_+)/(\rho_{\cL,-}(s_0) -\eps \sqrt{s_0} - \zeta_{-}).  \label{eq:satisfying-assumption-conds-3}
\eal

For condition (a), with $\log d/n \overset{n \to \infty}{\longrightarrow} 0$, $\eps = C_1 \sqrt{\log d/n}$, we have
$\eps \overset{n \to \infty}{\longrightarrow} 0$. It can be verified that $\kappa' \overset{n \to \infty}{\longrightarrow} \kappa$, and
(\ref{eq:satisfying-assumption-conds-1})-(\ref{eq:satisfying-assumption-conds-3})
holds with sufficiently large $n$ when Assumption~\ref{assumption:nonconvex-sparse-recovery}
holds.

For condition (b), with $\zeta_{-} = C_2\rho_{\cL,-}(s_0)$, $\eps\sqrt{s_0} \le C_3\rho_{\cL,-}(s_0)$, $C_2+C_3 < 1$ and $C_2 \le C'(1-C_3)$,
(\ref{eq:satisfying-assumption-conds-1})-(\ref{eq:satisfying-assumption-conds-2})
hold. Moreover, we have
\bals
\kappa' = \frac{\rho_{\cL,+}(s_0) + \eps \sqrt{s_0}- \zeta_+}{\rho_{\cL,-}(s_0) -\eps \sqrt{s_0} - \zeta_{-}}
\le \frac{(1+C_3)\rho_{\cL,+}(s_0)}{(1-C_2-C_3)\rho_{\cL,-}(s_0)}
 \le \frac{(1+C_3)(1+\delta)}{(1-C_2-C_3)(1-\delta)}  = \kappa_0.
\eals
As a result. $\tilde s > \pth{144 \kappa_0^2 + 250 \kappa_0} \bar s$ leads to (\ref{eq:satisfying-assumption-conds-3}).

\end{proof}

\begin{proof}
[\textup{\bf Proof of Theorem~\ref{theorem:sketching-nonconvex-minimax} }]
It follows by (\ref{eq:sketched-supnorm-gradient-L-barbeta}) in
Lemma~\ref{lemma:sketching-satisfying-assumption} and
Lemma~\ref{lemma:bounded-grad-L-barbeta} that
\bal\label{eq:sketching-nonconvex-minimax-seg1}
\supnorm{\nabla \tilde \cL(\bar \bbeta)}
\le \supnorm{\nabla \cL(\bar \bbeta)}  + \eps \sqrt{{\bar s}} \ltwonorm{\bar \bbeta} \le \pth{C_1 \sqrt{{\bar s}} \ltwonorm{\bar \bbeta}  +  2 \sigma} \sqrt{\frac{\log d}{n}}.
\eal
Let $\lambda_{\textup{tgt}} = 8 \pth{C_1 \sqrt{{\bar s}} \ltwonorm{\bar \bbeta}  +  2 \sigma} \sqrt{\frac{\log d}{n}}$. Because Assumption~\ref{assumption:nonconvex-sketch-sparse-recovery} holds, we can apply the approximate path following method described in
\cite[Algorithm 1]{Wang2014-sparse-nonconvex} to solve the original problem (\ref{eq:optimization-general}) and the sketched problem (\ref{eq:optimization-general-rp}) to obtain $\tilde \bbeta^*$ and $\tilde \bbeta^*$ such that
$\bbeta^*$ is an critical point of problem (\ref{eq:optimization-general}) and $\tilde \bbeta^*$
is an critical point of (\ref{eq:optimization-general-rp})
with $\tbX = \bP \bX$ and nonconvex regularizer $h_{\lambda_{\textup{tgt}}}$. Let $\cS = \supp{\bar \bbeta}$.
We can repeat the proof for \cite[Theorem 5.5]{Wang2014-sparse-nonconvex} and conclude that
\bals
\lzeronorm{\bbeta^*_{\cS^c}} \le \tilde s, \quad \lzeronorm{\tilde \bbeta^*_{\cS^c}} \le \tilde s.
\eals
The above inequalities show that $\abth{\supp{\tilde \bbeta^* - \bbeta^*} \bigcup
\supp{\bbeta^*} \bigcup \supp{\tilde \bbeta^*}} \le s_0= \bar s + 2 \tilde s$. It then follows by
Corollary~\ref{corollary::error-bound-noncovnex-sparse-eigenvalue} that
\bal\label{eq:sketching-nonconvex-minimax-seg2}
\ltwonorm{\tilde \bbeta^* - \bbeta^*}
\le \frac{\eps \sqrt{\rho_{\cL,+}(s_0)}}{(1-\eps)\rho_{\cL,-}(s_0) - \zeta_{-}  }  \norm{\bbeta^*}{\bX}.
\eal
In addition, it follows by \cite[Theorem 4.7, Theorem 4.8]{Wang2014-sparse-nonconvex} that the parameter estimation error of $\bbeta^*$ satisfies
\bal\label{eq:sketching-nonconvex-minimax-ori-rate}
\ltwonorm{\bbeta^* - \bar \bbeta} \le (21/8)/\pth{\rho_{\cL,-}(s_0)-\zeta_{-}}
\sqrt{\bar s}\lambda_{\textup{tgt}}.
\eal

It follows by (\ref{eq:sketching-nonconvex-minimax-ori-rate})
and (\ref{eq:sketching-nonconvex-minimax-seg2}) that
\bal\label{eq:sketching-nonconvex-minimax-seg3}
\ltwonorm{\tilde \bbeta^* - \bbeta^*}
&\le \frac{\eps \sqrt{(\bar s + \tilde s)\rho_{\cL,+}(s_0)}}{(1+C')/2 \cdot\rho_{\cL,-}(s_0) - \zeta_{-}  }  \ltwonorm{\bbeta^*} \nonumber \\
&\le \frac{\eps \sqrt{(\bar s + \tilde s)\rho_{\cL,+}(s_0)}}{(1+C')/2 \cdot\rho_{\cL,-}(s_0) - \zeta_{-}  } \ltwonorm{\bar \bbeta}
+ \frac{\eps \sqrt{(\bar s + \tilde s)\rho_{\cL,+}(s_0)}}{(1+C')/2 \cdot\rho_{\cL,-}(s_0) - \zeta_{-}  } \cdot (21/8)/\pth{\rho_{\cL,-}(s_0)-\zeta_{-}}
\sqrt{\bar s}\lambda_{\textup{tgt}} \nonumber \\
&\le\frac{\eps \sqrt{(\bar s + \tilde s)\rho_{\cL,+}(s_0)}}{(1+C')/2 \cdot\rho_{\cL,-}(s_0) - \zeta_{-}  } \ltwonorm{\bar \bbeta}
+ \frac {\sqrt{\bar s}\lambda_{\textup{tgt}}}{8\pth{\rho_{\cL,-}(s_0)-\zeta_{-}}}
\nonumber \\
&\le \frac{C_1\ltwonorm{\bar \bbeta} \sqrt{(1 + \tilde s/\bar s)\rho_{\cL,+}(s_0)}}{(1+C')/2 \cdot\rho_{\cL,-}(s_0) - \zeta_{-}  }  \sqrt{\frac{\bar s\log d}n}
+ \frac {\sqrt{\bar s}\lambda_{\textup{tgt}}}{8\pth{\rho_{\cL,-}(s_0)-\zeta_{-}}}
\eal

(\ref{eq:sketching-nonconvex-minimax}) then follows by (\ref{eq:sketching-nonconvex-minimax-ori-rate})
and (\ref{eq:sketching-nonconvex-minimax-seg3}).

\end{proof}

\begin{lemma}[{\cite[Lemma 5]{Zhang2010-sparse-regularization}}]
\label{lemma:bounded-grad-L-barbeta}
Under the linear model in Section~\ref{sec:sketching-sparse-signal-recovery}, that is, $\by = \bar \bX \bar \bbeta + \beps$ where $\beps$ is a noise vector of i.i.d. sub-gaussian elements with variance proxy $\sigma^2$. Let $\bX = \bar \bX/\sqrt{n}$, $\by = \bar \by/\sqrt{n}$, and $\cL(\bbeta) = 1/2\cdot \bbeta^{\top}  \bX^{\top} \bX \bbeta - \by^{\top} \bX \bbeta$. Then with probability at least $1-\eta$ for any $\eta \in (0,1)$,
\bal\label{eq:bounded-grad-L-barbeta}
\supnorm{\nabla L(\bar \bbeta)} \le \sqrt{2} \sigma \sqrt{\frac{\log (2d/\eta)}{n}}.
\eal

\end{lemma}

\begin{lemma}[{\cite[Proposition 5.20]{aubrun2017alice-bob-banach}}]
\label{lemma:concentration-lipschitz-unit-sphere}
Let $d > 2$. If $f \colon \unitsphere{d-1} \to \RR$ is a 1-Lipschitz function, then for any $t >0$,
\bal\label{eq:concentration-lipschitz-unit-sphere}
\Prob{\abth{f(\bx) - \Expect{}{f(\bx)}} > t} \le 2\exp\pth{-dt^2/2},
\eal
where $\bx$ is drawn uniformly from the unit sphere $\unitsphere{d-1}$ in $\RR^d$.
\end{lemma}


\end{document}